\documentclass[12pt]{amsart}
\usepackage[colorlinks=true,citecolor=magenta,linkcolor=magenta]{hyperref}
\usepackage{amssymb,verbatim,amscd,amsmath,graphicx}
\usepackage{enumerate}
\usepackage{graphicx}
\usepackage{caption}
\usepackage{color,colortbl}
\usepackage{fullpage}
\usepackage{bbm}
\usepackage{comment}

\usepackage{pgf,tikz,pgfplots}
\usepackage{mathrsfs}
\usetikzlibrary{arrows}

\numberwithin{equation}{section}
\newtheorem{theorem}{Theorem}[section]
\newtheorem{lemma}[theorem]{Lemma}
\newtheorem{proposition}[theorem]{Proposition}
\newtheorem{corollary}[theorem]{Corollary}

\theoremstyle{definition}
\newtheorem{definition}[theorem]{Definition}

\newtheorem{def-prop}[theorem]{Definition-Proposition}

\newtheorem{remark}[theorem]{Remark}
\newtheorem{example}[theorem]{Example}

\newtheorem*{acknowledgement}{Acknowledgements}

\newtheorem{question}[theorem]{Question}

\DeclareMathOperator{\Ass}{Ass}
\DeclareMathOperator{\Min}{Min}

\DeclareMathOperator{\lcm}{lcm}

\DeclareMathOperator{\mdc}{mdc}
\DeclareMathOperator{\NP}{NP}
\DeclareMathOperator{\SP}{SP}
\DeclareMathOperator{\conv}{convex\ hull}
\DeclareMathOperator{\sgt}{sgt}
\DeclareMathOperator{\svd}{svd}

\newcommand{\ZZ}{{\mathbb Z}}
\newcommand{\NN}{{\mathbb N}}
\newcommand{\QQ}{{\mathbb Q}}
\newcommand{\RR}{{\mathbb R}}

\newcommand{\kk}{{\mathbbm k}}

\def\I{{\mathcal I}}
\def\J{{\mathcal J}}

\def\A{{\mathcal A}}

\def\C{{\mathcal C}}

\def\R{{\mathcal R}}

\def\R{{\mathcal R}}

\def\mm{{\mathfrak m}}

\def\pp{{\mathfrak p}}

\def\a{{\bf a}}
\def\b{{\bf b}}

\def\p{{\bf p}}
\def\q{{\bf q}}

\def\1{{\bf 1}}
\def\0{{\bf 0}}

\begin{document}
	
\title{Newton-Okounkov body, Rees algebra and analytic spread of graded families of monomial ideals}

\author{Huy T\`ai H\`a}
\address{Tulane University \\ Department of Mathematics \\
	6823 St. Charles Ave. \\ New Orleans, LA 70118, USA}
\email{tha@tulane.edu}

\author{Th\'ai Th\`anh Nguy$\tilde{\text{\^E}}$n}
\address{McMaster University, Department of Mathematics and Statistics,
	1280 Main Street West, Hamilton, Ontario, Canada \\
	and University of Education, Hue University, 34 Le Loi St., Hue, Viet Nam}
\email{nguyt161@mcmaster.ca}

\keywords{Newton-Okounkov body, Newton polyhedron, symbolic polyhedron, graded family of ideals, monomial ideals, analytic spread, Rees algebra, Noetherian Rees algebra}
\subjclass[2020]{13A30, 05E16, 05E40}

\begin{abstract}
Let $\I =\{I_k\}_{k \in \NN}$ be a graded family of monomial ideals. We use the Newton-Okounkov body of $\I$ to: (a) give a characterization for the Noetherian property of the Rees algebra of the family $\I$; and (b) present a combinatorial interpretation for the analytic spread of $\I$. We also apply these results to investigate and give bounds for the generation type and the Veronese degree of the symbolic Rees algebra of a monomial ideal.
\end{abstract}

\maketitle



\section{Introduction} \label{sec.intro}

This paper investigates properties of the Newton-Okounkov bodies associated to graded families of monomial ideals and their algebraic consequences.
Convex bodies have always played a special role in the study of algebraic properties and invariants. Their uses have led to many exciting results and applications in various areas of mathematics; for instance, in commutative algebra, in toric geometry and in tropical geometry (cf. \cite{CLS2011, MS2005, MS2015}).

A classical formula (see \cite{Teissier1988}) in multiplicity theory interprets the Hilbert-Samuel multiplicity of a monomial ideal, that is primary to the maximal homogeneous ideal, in terms of the volume of the complement of its Newton polyhedron (see also \cite{JM2013}). In \cite{CEHH2017}, the symbolic polyhedron of a monomial ideal was introduced to encode the asymptotic data of symbolic powers in a similar way that the Newton polyhedron does its ordinary powers. Newton and symbolic polyhedra of monomial ideals and their relationship continue to inspire much of current research (cf. \cite{BA2018,B+2021,C+2021,DFMS2019}). For example, a natural question is: for a squarefree monomial ideal, when are its Newton and symbolic polyhedra coincide? --- This question has a fascinating connection to a long standing conjecture in combinatorial optimization, namely, the Conforti-Cornu\'ejols conjecture (cf. \cite{D+2018, HT2019}).

Let $\kk$ be a field, let $R = \kk[x_1, \dots, x_n]$ be a polynomial ring over $\kk$, and let $I \subseteq R$ be a monomial ideal. We recall that the \emph{Newton} and \emph{symbolic polyhedra} of $I$ are defined as follows:
\begin{align*}
	\NP(I) & = \conv\left(\left\{\a \in \NN^n ~\big|~ x^\a \in I\right\}\right), \text{ and } \\
	\SP(I) & = \bigcap_{\pp \in \text{maxAss}(I)} \NP(Q_{\subseteq \pp}).
\end{align*}
Here, $\text{maxAss}(I)$ denotes the set of maximal associated primes of $I$, and $Q_{\subseteq \pp} = R \cap IR_\pp$. 

Both the Newton and symbolic polyhedra turn out to be particular cases of the Newton-Okounkov body (see Remark \ref{rmk.NPSP}).
The notion of Newton-Okounkov body was systematically introduced by Lazarsfeld and Musta\c{t}\v{a} \cite{LM2009} and by Kaveh and Khovanskii \cite{KK2012, KK2014}. The term \emph{Newton-Okounkov} was dedicated to Okounkov's pioneering work \cite{Ok1996, Ok2003}, in which the asymptotic multiplicities of irreducible representations, in the homogeneous coordinate ring of projective variety with a reductive group action, are interpreted as volumes of certain convex bodies. This construction has attracted much attention from researchers from across different areas, including algebraic and convex geometry and commutative algebra (cf. \cite{CM2021a,CM2021b,Cut2013,Cut2014,FH2021,Kav2015,NN2017,RW2019,Roe2016} and references therein thereafter).

For algebraic applications, the Newton-Okounkov body has been defined often for a \emph{graded algebra of integral type} or a graded family of ideals that are \emph{primary to the maximal homogeneous ideals}.
Our interest is in graded families of \emph{monomial} ideals. More specifically, let $\I = \{I_k\}_{k \in \NN}$ be a graded family of monomial ideals in $R$. As in \cite[Definition 4.7]{KK2014}, the following convex region is associated to $\I$.
$$\Delta(\I) = \overline{ \bigcup_{k \in \NN} \left\{\dfrac{\a}{k}  ~\big|~ x^\a \in I_k\right\}} \subseteq \RR^n.$$
In similar constructions for graded algebras of integral type, Kaveh and Khovanskii \cite{KK2012} called the \emph{compact} complements of the corresponding convex regions their \emph{Newton-Okounkov bodies}, while Cid-Ruiz, Mohammadi and Monin \cite{CMM2021} referred to these convex regions as their \emph{global Newton-Okounkov bodies}. For a graded family $\I$ of monomial ideals, the complement of $\Delta(\I)$ in $\RR^n_{\ge 0}$ is not necessarily compact. Thus, to avoid unnecessarily complicated terms and following Cutkosky's terminology \cite{Cut2013, Cut2014}, we shall call $\Delta(\I)$ the \emph{Newton-Okounkov body} of $\I$.

This paper aims to add important evidences to the rich and exciting connection between Newton-Okounkov bodies and algebraic properties and invariants of graded families of ideals. Particularly, we shall use combinatorial data of the Newton-Okounkov body of a graded family of monomial ideals to study the \emph{Noetherian property} of the \emph{Rees algebra} and to describe the \emph{analytic spread} of the given family.

Determining when the Rees algebra of a graded family of ideals is Noetherian is a difficult problem and, in general, is out of reach. It is closely related to Nagata's counterexample to Hilbert's fourteen problem (see \cite{Nag1959}). Many examples also exist to show that the Rees algebra of the family of symbolic powers of an ideal needs not be Noetherian (cf. \cite{Cut1991, Hun1982, Rob1985}). Our first main result characterizes the Noetherian property of the Rees algebra $\R(\I)$ of a graded family $\I$ of monomial ideals via its Newton-Okounkov and \emph{limiting} bodies.

\medskip

\noindent\textbf{Theorem \ref{thm.Noetherian}.} Let $\I = \{I_k\}_{k \in \NN}$ be a graded family of monomial ideals in $R$ and let $\overline{\I} = \{\overline{I_k}\}_{k \in \NN}$. The following are equivalent:
\begin{enumerate}
	\item There exists an integer $c$ such that $\Delta(\I) = \dfrac{1}{c}\NP(I_c)$.
		\item The limiting body ${\displaystyle \C(\I) = \bigcup_{k \in \NN}\dfrac{1}{k}\NP(I_k)}$ is a polyhedron.
	\item $\R(\overline{\I})$ is Noetherian.
	\item $\R(\I)$ is Noetherian.
\end{enumerate}

In initial studies of Newton-Okounkov bodies, a particular result that has captivated many algebraists' interest is the beautiful ``Volume $=$ Multiplicity'' formula (see, for instance, \cite[Theorem A]{LM2009}, \cite[Theorem 2]{KK2012} and \cite[Theorem 6.8]{KK2014}). Since then this formula has been explored and generalized in various directions (cf. \cite{Cut2013,Cut2014,CM2021a,CM2021b}). We shall provide yet another formula, which interprets the analytic spread of a graded family of monomial ideals in terms of its Newton-Okounkov body.

The \emph{analytic spread} of a graded family $\I$ of ideals in $R$ is defined to be
$$\ell(\I) = \dim (\R(\I)/\mm\R(\I)),$$
where $\mm$ denotes the maximal homogeneous ideal in $R$.
This notion of $\ell(\I)$ generalizes the familiar notion of the \emph{analytic spread} $\ell(I)$ of an ideal $I$ (when $\I$ is the family of ordinary powers of $I$) and the \emph{symbolic analytic spread} $\ell_s(I)$ (when $\I$ is the family of symbolic powers of $I$). Geometrically, the analytic spread of an ideal $I$ is the dimension of the special fiber of the blowup along the variety defined by $I$, while the symbolic analytic spread of a locally principal ideal $I$ on the
punctured spectrum can be viewed as a local version of the Kodaira dimension associated to this divisor if the symbolic Rees algebra is Noetherian. A recent work of Cutkosky and Sarkar \cite[Lemma 3.6]{CS2021} showed that for any graded \emph{filtration} $\I$ of ideals in a local ring, $\ell(\I)$ is finite. Their argument in fact extends to any graded family of ideals in a polynomial ring; see, for instance, Proposition \ref{lem.CS3.6}.

When $\I$ is Noetherian (i.e., the Rees algebra $\R(\I)$ is a Noetherian ring), $\ell(\I)$ also provides a measure for the rate of growth of the minimum number of generators of $\I$. Specifically, by letting $\mu(I)$ denote the minimum number of generators for an ideal $I$ and assuming that $\R(\I)$ is Noetherian, it is observed in Remark \ref{rmk.ellNoetherian} that
\begin{align}
\ell(\I) = \min\{t \in \RR ~\big|~ \mu(I_k) = O(k^{t-1})\}.\label{eq.ell*}
\end{align}
This equality is not known if $\I$ is not necessarily Noetherian. For families of symbolic powers of ideals, the right hand side of (\ref{eq.ell*}), which is denoted by $\ell^*(\I)$, has been studied, for example, in \cite{BS1990,DM2020,Dut1983,HKTT2017}, but it is not known if this invariant is finite in general.

Let $\mdc(\Delta)$ represent the \emph{maximum dimension of a compact face} of a polyhedron $\Delta$. Our next main result reads as follows.

\medskip

\noindent{\textbf{Theorem \ref{thm.analyticSpread}.}} Let $\I$ be a Noetherian graded family of monomial ideals in $R$. Let $\Delta(\I)$ and $\ell(\I)$ be its Newton-Okounkov body and analytic spread, respectively. Then,
	$$\ell(\I) = \mdc(\Delta(\I))+1.$$

A particular application of Theorem \ref{thm.analyticSpread} is when $\I = \{I^{(k)}\}_{k \in \NN}$ is the family of symbolic powers of a monomial ideal $I$. In this case, Theorem \ref{thm.analyticSpread} generalizes the combinatorial interpretation of the ordinary analytic spread of a monomial ideal given by Bivi\`a-Ausina \cite{BA2003} to the symbolic analytic spread; see Corollary \ref{cor.ellsG}.

For the less understood invariant $\ell^*(\I)$, we obtain a similar result to Theorem \ref{thm.analyticSpread} in a slightly more general setting, where the Newton-Okounkov body is defined with respect to a \emph{good} valuation; see \cite{Cut2013, KK2014} for the definition and existence of good valuations. The proof of this result provides an instance when one could study families of non-monomial ideals by reducing to those of monomial ideals.

\medskip

\noindent\textbf{Theorem \ref{thm.ell*valuation}.} Let $\I$ be a graded family of $\mm-$primary homogeneous ideals, and let $v$ be a good valuation that respects the monomials in $R$. Let $\C(\I)$ and $\Delta(\I)$ be the limiting and Newton-Okounkov bodies of $\I$ defined by $v$. Suppose that $\C(\I)$ is a polyhedron. Then, we have
$$\ell^*(\I) = \mdc(\Delta(\I)) + 1 = \dim R.$$

Specializing our results on graded families of monomial ideals to the ordinary powers and symbolic powers of an ideal gives to, in the context of the Newton and symbolic polyhedra, our investigation on the Newton-Okounkov body reveals an interesting connection between these two polyhedra associated to a monomial ideal. We prove the following theorem.

\medskip

\noindent\textbf{Theorem \ref{thm.NPSP}.} Let $I \subseteq R$ be a monomial ideal. There exists an integer $c$ such that
$$\NP(I^{(c)}) = c\SP(I).$$

The constant $c$ in Theorem \ref{thm.NPSP} is given implicitly in Theorem \ref{thm.Noetherian}. However, we show that, when $I$ is a squarefree monomial ideal, $c$ can be obtained explicitly from the vertices of the symbolic polyhedron $\SP(I)$ of $I$; see Theorem \ref{thm.c}.

Our description of the constant $c$ in Theorem \ref{thm.c} can further be used to give bounds for the symbolic generation type and Veronese degree of a squarefree monomial ideal. It is known (cf. \cite[Theorem 3.2]{HHT2007}) that the symbolic Rees algebra $\R_s(I)$ is finitely generated. Particularly, this ensures that the maximum generating degree $\sgt(I)$ of $\R_s(I)$ and the smallest degree $\svd(I)$ such that the $\svd(I)$-th Veronese subring of $\R_s(I)$ is standard graded are both finite. The invariants $\sgt(I)$ and $\svd(I)$ are called the \emph{symbolic generation type} and the \emph{symbolic Veronese degree} of $I$. These are classical objects of study in commutative algebra that have been much investigated (see, for instance, \cite{Cut1991, HHT2007, HHTZ2008,GS2021} and references therein). We obtain the following results that provide better bounds (and sharp bounds in some cases) than previously known bounds for those invariants.

\medskip

\noindent\textbf{Theorems \ref{thm.svd} and \ref{thm.GT}.} Let $I$ be a squarefree monomial ideal and suppose that $\{v_1, \dots, v_r\}$ are the vertices of $\SP(I)$. Let $d_i$ be the least common multiple of the denominators of coordinates of $v_i$, for $i = 1, \dots, r$. Set $D = \max\{d_1, \dots, d_r\}$ and $c = \lcm(d_1, \dots, d_r).$ Then,
\begin{enumerate}
	\item $c \le \svd(I) \le (\ell_s(I)-1)c = \mdc(\SP(I))c$; and
	\item $\sgt(I) \le \max\{\ell_s(I)D-1,D\}$.
\end{enumerate}

The paper is outlined as follows. In the next section, we collect important notations and terminology used in the paper. Particularly, we shall recall the definition of the Newton and symbolic polyhedra of a monomial ideal, define the Newton-Okounkov body and analytic spread of a graded family of monomial ideals, and discuss their basic properties.

In Section \ref{sec.NO}, we use the Newton-Okounkov body to characterize the Noetherian property of the Rees algebra. Our first main result, Theorem \ref{thm.Noetherian}, is proved in this section. To prove this result, we first establish the equivalence between $\Delta(\I) = \frac{1}{c}\NP(I_c)$ and $\C(\I)$ being a polyhedron -- this is done in Theorem \ref{lem.NPpoly}. For the remaining equivalences, we observe that the Noetherian property of $\R(\I)$ and its Veronese subrings are equivalent. This allows us to focus on Veronese subrings of $\R(\I)$ and conditions of the form $\frac{1}{c}\NP(I_c) = \frac{1}{kc}\NP(I_{kc})$ for a fixed $c$ and $k \in \NN$. We then show that the Noetherian property of $\R(\I)$ and $\R(\overline{\I})$ are equivalent. This is reflected through the fact that $\Delta(\I) = \Delta(\overline{\I})$, as shown in Proposition \ref{lem.int}. However, the proof requires a special care.

In section \ref{sec.AS}, we present a formula to compute the analytic spread of a graded family of monomial ideals in terms of the maximum dimension of a compact face of its Newton-Okounkov body. Our next main result, Theorem \ref{thm.analyticSpread}, is proved in this section. To prove this theorem, we make use of Theorem \ref{thm.Noetherian}, and notice that when $\Delta(\I) = \frac{1}{c}\NP(I_c)$, the maximum dimension of a compact face of $\Delta(\I)$ is the same as that of $\NP(I_c)$. The known result of Bivi\`a-Ausina \cite{BA2003} can now be applied for $I_c$. Theorem \ref{thm.ell*valuation} provides a similar result to that of Theorem \ref{thm.analyticSpread} in more general setting. We also show in this section that the arguments in \cite{CS2021} apply to show that $\ell(\I)$ is a finite invariant; Proposition \ref{lem.CS3.6}.

Section \ref{sec.NPSP} focuses on the Newton and symbolic polyhedra as special cases of the Newton-Okounkov body. Theorem \ref{thm.NPSP} is proved in this section. To prove Theorem \ref{thm.NPSP}, we observe that when $\I = \{I^{(k)}\}_{k \in \NN}$ is the family of symbolic powers of a monomial ideal $I$, $\SP(I) = \Delta(\I)$ and the symbolic Rees algebra $\R_s(I)$ is Noetherian. We then make use of Theorems \ref{lem.NPpoly} to relate the $\Delta(\I)$ to $\NP(I^{(c)})$. The explicit combinatorial description for this constant $c$ is also given in this section; Theorem \ref{thm.c}.

Finally, in Section \ref{sec.Veronese}, we provide bounds for the symbolic relation type and Veronese degree for a monomial ideal $I$, making use of the constant $c$ described in Theorem \ref{thm.NPSP}. Theorems \ref{thm.svd} and \ref{thm.GT} are proved in this section. To prove Theorem \ref{thm.svd}, we observe that the symbolic Veronese degree $d = \svd(I)$ is an integer satisfying the condition that $I^{(dk)} = \left(I^{(d)}\right)^{(k)}$ for any $k \in \NN$. We then invoke Lemma \ref{lem.int=symb}, which characterizes the equality $\NP(I) = \SP(I)$, and Theorem \ref{thm.c} to show that $\svd(I) \ge c$. The other inequality of Theorem \ref{thm.svd} is established by making use of \cite[Theorem 5.1]{Sin2007} which shows that the reduction number of $I$ is at most $\ell(I)-1$. To prove Theorem \ref{thm.GT}, we examine the combinatorial construction of the symbolic polyhedron $\SP(I)$ and use the Caratheodory theorem. 

\begin{acknowledgement} The first author is partially supported by Louisiana Board of Regents and Simons Foundation.
\end{acknowledgement}


 \section{Preliminaries} \label{sec.prel}

 In this section, we collect important definitions and notations used in the paper. For unexplained terminology from algebra and convex geometry we refer the interested reader to the following texts \cite{BH1993, MS2005, Zie1995}.

 Throughout the paper, $\kk$ denotes a field, and $R = \kk[x_1, \dots, x_n]$ is a polynomial ring over $\kk$. For $\a = (a_1, \dots, a_n) \in \NN^n$, let $x^\a$ denote the monomial $x_1^{a_1} \cdots x_n^{a_n}$ in $R$. By a \emph{polyhedron} in $\RR^n$ we mean the intersection of finitely many closed half spaces. Particularly, a polyhedron is a closed convex set.

 \begin{definition} \label{def.mdc}
 	For a polyhedron $P$, let $\mdc(P)$ denote the maximum dimension of a compact face of $P$.
 \end{definition}

\begin{definition} \label{def.symbolic}
	Let $I \subseteq R$ be an ideal and let $\Ass(I)$ denote the set of its associated primes. For any $k \in \NN$, the $k$-th \emph{symbolic power} of $I$ is defined to be
	$$I^{(k)} = \bigcap_{\pp \in \Ass(I)} \left(I^k R_\pp \cap R\right).$$
\end{definition}

We remark here that there is also a notion of symbolic powers in which the set $\Min(I)$ of minimal primes is used in place of the set $\Ass(I)$ of associated primes in Definition \ref{def.symbolic}. For squarefree monomial ideals or, more generally, ideals with no embedded primes, these two notions of symbolic powers agree.

 \begin{definition}
 	\label{def.NPSP}
 	Let $I \subseteq R$ be a monomial ideal.
 	\begin{enumerate}
 		\item (See \cite{HS2006}) The \emph{Newton polyhedron} of $I$ is defined to be
 		$$\NP(I) = \conv\left(\left\{ \a \in \NN^n ~\big|~ x^\a \in I\right\}\right).$$
 		\item (See \cite{CEHH2017}) Let $\text{maxAss}(I)$ denote the set of maximal associated primes of $I$ and set $Q_{\pp} = R \cap IR_\pp$ for $\pp \in \text{maxAss}(I).$ The \emph{symbolic polyhedron} of $I$ is defined to be
 		$$\SP(I) = \bigcap_{\pp \in \text{maxAss}(I)} \NP(Q_{\subseteq \pp}).$$
 	\end{enumerate}
 \end{definition}

It is easy to see that both $\NP(I)$ and $\SP(I)$ are \emph{rational} polyhedra in $\RR^n$.
The Newton and symbolic polyhedra are of particular interest to us thanks to the following simple membership criteria.

\begin{lemma} \label{lem.membership}
	Let $I \subseteq R$ be a monomial ideal and let $\a \in \ZZ^n_{\ge 0}$.
\begin{enumerate}
	\item (See, for example, \cite[Proposition 2.7]{DFMS2019}) For any given integer $k$, we have
	$$x^\a \in \overline{I^k} \text{ if and only if } \dfrac{\a}{k} \in \NP(I).$$
	\item (See, for example, \cite[Proposition 2.10]{DFMS2019}) Assume, in addition, that $I$ is squarefree. For any given integer $k$, we have
	$$x^\a \in I^{(k)} \text{ if and only if } \dfrac{\a}{k} \in \SP(I).$$
\end{enumerate}
\end{lemma}

When $I$ is a squarefree monomial ideal, $I$ has the form $I = \bigcap_{\pp \in \Min(I)} \pp$, where all associated primes of $I$ are \emph{monomial prime} ideals (i.e., generated by a collection of the variables). In this case, $\SP(I)$ has a simple description:
$$\SP(I) = \bigcap_{\pp \in \Min(I)} \{(a_1, \dots, a_n) \in \RR^n_{\ge 0} ~\big|~ \sum_{x_i \in \pp} a_i \ge 1\}.$$
This simple description for $\SP(I)$ holds also for a slightly larger class of monomial ideals.

\begin{definition} \label{def.linear-power}
	We say that a monomial ideal $I$ is of \emph{linear-power type} if $I$ can be written as
$$I = \bigcap_{\pp \in \Min(I)} \pp^{\omega_\pp},$$
in which $\omega_\pp \in \NN$ and all minimal primes of $I$ are monomial primes.
\end{definition}

For a monomial ideal of linear-power type, $I = \bigcap_{\pp \in \Min(I)} \pp^{\omega_\pp}$, we have
\begin{align}
	\SP(I) = \bigcap_{\pp \in \Min(I)} \{(a_1, \dots, a_n) \in \RR^n_{\ge 0} ~\big|~ \sum_{x_i \in \pp} a_i \ge \omega_\pp\}. \label{eq.SP}
\end{align}

\begin{remark}
	\label{rmk.SP}
	It is easy to see that Lemma \ref{lem.membership}.(2) holds also for monomial ideals of linear-power type.
\end{remark}

Note further that Lemma \ref{lem.membership}.(2) may not hold true if the ideal $I$ is not of linear-power type, as illustrated in the following example.

\begin{example}
	\label{ex.SPmonomial}
	Let $I = (x^4, xy^2, y^3) \subseteq \kk[x,y]$. The ideal $I$ is $(x,y)$-primary. Thus,
	$$\SP(I) = \NP(I).$$
	Observe that $(3,1) \in \NP(I) = \SP(I)$ since, for instance, $x^3y \in \overline{I}$. However, $x^3y \not\in I^{(1)} = I$.
\end{example}

Lemma \ref{lem.membership} shows that $\NP(I^k) = k\NP(I)$, and if in addition $I$ is squarefree (or, more generally, of linear-power type) then
\begin{enumerate}
	\item $\mathcal{L}(I^{(k)}) = k\SP(I) \cap \ZZ^n_{\ge 0},$ where $\mathcal{L}(J)$ denotes the lattice of exponents of monomials inside a monomial ideal $J$, and
	\item $\SP(I^{(k)}) = k\SP(I).$
\end{enumerate}

Observe that, by definition, $\NP(I) \subseteq \SP(I)$. It follows that, for a squarefree monomial ideal $I$ (or, more generally, a linear-power type monomial ideal $I$), we have
$$\overline{I^k} \subseteq I^{(k)} \text{ for all } k \in \NN.$$
This later containment, however, does not necessarily hold if $I$ is not of linear-power type, as also illustrated by Example \ref{ex.SPmonomial}.

We now define the main objects of study in this paper, the Newton-Okounkov body of a graded family of monomial ideals.

 \begin{definition} \label{def.family}
 	A collection $\I = \{I_k\}_{k \in \NN}$ of ideals in $R$ is called a \emph{graded family} if $I_p \cdot I_q \subseteq I_{p+q}$ for all $p, q \in \NN$. A graded family of ideals is called a \emph{filtration} if, in addition, we have $I_p \supseteq I_{p+1}$ for all $p \in \NN$.
 \end{definition}

 \begin{definition}[\protect{\cite[Definition 4.7]{KK2014}}]
 	\label{def.NO}
 	Let $\I = \{I_k\}_{k \in \NN}$ be a graded family of monomial ideals in $R$. The \emph{Newton-Okounkov body} of $\I$ is defined to be
 	$$\Delta(\I) = \overline{\bigcup_{k \in \NN} \left\{ \dfrac{\a}{k} ~\big|~ x^\a \in I_k \right\}} \subseteq \RR^n.$$
 \end{definition}

 Our definition of the Newton-Okounkov body (Definition \ref{def.NO}) only works for monomial ideals. Although this is slightly different from the notion of the Newton-Okounkov body associated to a \emph{graded algebra of integral type} and a \emph{faithful} valuation that was investigated in \cite{KK2012}, it is, in fact, isomorphic to the Newton-Okounkov body associated to the semigroup of all exponent vectors of the \textit{Rees algebra} of $\I$.

\begin{remark}
	\label{rmk.DeltaNP}
	It can be shown that $\Delta(\I)$ is a convex set. Thus, it follows from the definition that
	$\Delta(\I) =  \overline{\bigcup_{k \in \NN} \dfrac{1}{k}\NP(I_k)}.$
	In particular, $\Delta(\I)$ is the closure of the \emph{limiting body}
	${\displaystyle \C(\I) := \bigcup_{k \in \NN} \dfrac{1}{k} \NP(I_k)}$
	of $\I$, that was introduced and investigated recently in \cite{C+2021}; this asymptotic object was also discussed and studied in various contexts in \cite{Wol2008,Mayes14a,Mayes14b}.
\end{remark}

 \begin{remark}
 	\label{rmk.NPSP}
 	Observe that if $\I = \{I^k\}_{k \in \NN}$ is the family of ordinary powers of a monomial ideal $I$ then, by Remark \ref{rmk.DeltaNP}, we have
 	$$\Delta(\I) = \overline{\bigcup_{k \in \NN} \dfrac{1}{k} \NP(I^k)} = \NP(I).$$
 	
 	On the hand, if $\I = \{I^{(k)}\}_{k \in \NN}$ is the family of symbolic powers of a monomial ideal $I$ then, by \cite[Corollary 3.12]{C+2021}, we have
 		$\SP(I) = \bigcup_{k \in \NN} \dfrac{1}{k}\NP(I^k).$
 	Since $\SP(I)$ is a closed subset of $\RR^n$, we then have
 	$$\SP(I) = \Delta(\I).$$
 	Hence, the Newton and symbolic polyhedra of monomial ideals are particular cases of the Newton-Okounkov body.
 \end{remark}

\begin{definition}
    \label{def.Rees}
	Let $\I = \{I_k\}_{k \in \NN}$ be a graded family of ideals in $R$. The \emph{Rees algebra} of $\I$ is defined to be
$$\R(\I) = \bigoplus_{k \ge 0} I_k t^k \subseteq R[t].$$
\end{definition}

When $\I = \{I^k\}_{k \in \NN}$ or $\I = \{I^{(k)}\}_{k \in \NN}$ is the family of ordinary powers or the family of symbolic powers of an ideal $I \subseteq R$, the Rees algebra $\R(\I)$ is the familiar (and yet still far from being well understood) Rees algebra $\R(I)$ or \emph{symbolic} Rees algebra $\R_s(I)$ of $I$.

\begin{definition}
	\label{def.AS}
	Let $\I = \{I_k\}_{k \in \NN}$ be a graded family of ideals and let $\mm = (x_1, \dots, x_n)$ be the maximal homogeneous ideal in $R$. The \emph{analytic spread} of $\I$ is defined to be
	$$\ell(\I) = \dim (\R(\I)/\mm\R(\I)) = \dim (\R(\I) \otimes_R R/\mm).$$
\end{definition}

As with the Rees algebra, the analytic spread of a graded family of ideals generalizes the analytic spread and the symbolic analytic spread of an ideal. In general, it is not know if $\ell(\I)$ is finite. It is a recent result of Cutkosky and Sarkar (see \cite[Lemma 3.6]{CS2021}) that if $\I$ is a filtration then $\ell(\I)$ is finite (and bounded above by $\dim R$). The ring $\R(\I) \otimes_R R/\mm$ is often referred to as the \emph{special fiber ring} of $\I$.

A graded family $\I$ is called \emph{Noetherian} if its Rees algebra is. In general, the Rees algebra of a graded family $\I$ needs not be Noetherian. This is also the case even for the symbolic Rees algebra of an ideal or when the family consists of monomial ideals.

The symbolic Rees algebra of a monomial ideal is known to be Noetherian (see \cite[Theorem 3.2]{HHT2007}). Thus, the following invariants are finite.

\begin{definition}
	Let $I$ be a monomial ideal in $R$.
	\begin{enumerate}
		\item The \emph{symbolic generation type} of $I$ is defined to be the maximum generating degree of $\R_s(I)$. That is,
		$$\sgt(I) := \inf \{d ~\big|~ \R_s(I) = R[It, I^{(2)}t^2, \dots, I^{(d)}t^d]\}.$$
		\item The \emph{symbolic Veronese degree} of $I$ is defined to be the smallest integer $d$ such that the $d$-th Veronese subring of $\R_s(I)$ is standard graded. That is,
		$$\svd(I) := \inf \{d ~\big|~ I^{(dk)} = \left(I^{(d)}\right)^k \text{ for all } k \in \NN\}.$$
	\end{enumerate}
\end{definition}
These invariants are sometimes referred to as simply the generation type and the standard Veronese degree of the \emph{symbolic} Rees algebra of $I$.


 \section{Newton-Okounkov body and the Noetherian property of Rees algebra} \label{sec.NO}

 In this section, we investigate the Newton-Okounkov body associated to a graded family of monomial ideals, and its role in determining the Noetherian property of the Rees algebra of the given family.

We start by relating properties of the Newton-Okounkov and limiting bodies of a graded family of monomial ideals to the Newton polyhedron of a single ideal in the family.

  \begin{theorem}
 	\label{lem.NPpoly}
 	Let $\I = \{I_k\}_{k \in \NN}$ be a graded family of monomial ideals in $R$. The following are equivalent:
 	\begin{enumerate}
 		\item The limiting body ${\displaystyle \C(\I) = \bigcup_{k \in \NN}\dfrac{1}{k}\NP(I_k)}$ is a polyhedron.
 		\item There exists an integer $c$ such that $\Delta(\I) = \dfrac{1}{c}\NP(I_c)$.
 		\item There exists an integer $c$ such that $\dfrac{1}{c}\NP(I_c) = \dfrac{1}{kc}\NP(I_{kc})$ for all $k \in \NN$.
 	\end{enumerate}
 \end{theorem}

 \begin{proof} We start by showing that $(1) \Leftrightarrow (2)$. If there exists an integer $c$ such that $\Delta(\I) = \frac{1}{c}\NP(I_c)$ then we have
 	$$\dfrac{1}{c}\NP(I_c) \subseteq \bigcup_{k \in \NN} \dfrac{1}{k}\NP(I_k) \subseteq \Delta(\I) = \dfrac{1}{c}\NP(I_c).$$
 	It follows that $\bigcup_{k \in \NN} \frac{1}{k}\NP(I_k) = \frac{1}{c}\NP(I_c)$ is a polyhedron. Thus, $(2) \Rightarrow (1)$.
 	
 	Conversely, suppose that $P = \bigcup_{k \in \NN} \NP(I_k)$ is a polyhedron. Let $v_1, \dots, v_m$ be the vertices of $P$ and suppose that $v_i \in \frac{1}{c_i}\NP(I_{c_i})$. Set $c = c_1 \dots c_m$. It is easy to see that
 	$$cv_i \in \prod_{j \not= i}c_j \NP(I_{c_i}) \subseteq \NP(I_{c_i \cdot \prod_{j \not= i}c_j}) = \NP(I_c).$$
 	Therefore, $v_i \in \frac{1}{c}\NP(I_c)$ for all $i = 1, \dots, m$. This implies that $P \subseteq \frac{1}{c}\NP(I_c)$. Hence, $P = \frac{1}{c}\NP(I_c)$ and, as a consequence,
 	$$\Delta(\I) = \overline{P} = \frac{1}{c}\NP(I_c).$$
 	
 	We shall now prove that $(2) \Leftrightarrow (3)$. Observe that for any $k, c \in \NN$, we have
 	\begin{align}
 		k\NP(I_c) \subseteq \NP(I_{kc}). \label{eq.filtration}
 	\end{align}
 	Thus, $\frac{1}{c}\NP(I_c) \subseteq \frac{1}{kc}\NP(I_{kc})$. This shows that the implication $(2) \Rightarrow (3)$ is obvious.
 	
 	Conversely, suppose there exists an integer $c$ such that $\frac{1}{c}\NP(I_c) = \frac{1}{kc}\NP(I_{kc})$ for all $k \in \NN$. Consider an arbitrary point $v \in \bigcup_{k \in \NN} \frac{1}{k}\NP(I_k)$. Let $k_0 \in \NN$ be such that $v \in \frac{1}{k_0}\NP(I_{k_0})$. As in (\ref{eq.filtration}), we have
 	$$\frac{1}{k_0}\NP(I_{k_0}) \subseteq \frac{1}{k_0c}\NP(I_{k_0c}) = \frac{1}{c}\NP(I_c).$$
 	This is true for any $v \in \bigcup_{k \in \NN} \frac{1}{k}\NP(I_k)$, so we have $\bigcup_{k \in \NN} \frac{1}{k}\NP(I_k) \subseteq \frac{1}{c}\NP(I_c)$. It follows that $\bigcup_{k \in \NN} \frac{1}{k}\NP(I_k) = \frac{1}{c}\NP(I_c)$ is a closed convex set. Hence, $\Delta(\I) = \frac{1}{c}\NP(I_c)$. We have established $(3) \Rightarrow (2)$, and the proof completes.
 \end{proof}

\begin{example} \label{ex.nonpolyDelta}
	For a general graded family of monomial ideals, the conditions in Theorem \ref{lem.NPpoly} may not hold. In fact, the Newton-Okounkov body of a graded family of monomial ideals needs not be polyhedral. It was shown in \cite[Proposition 2]{Wol2008} that for any non-empty closed convex set $P \subseteq \RR^n_{\ge 0}$ absorbing $\RR^n_{\ge 0}$, i.e., $P + \RR^n_{\ge 0} \subseteq P$, there exists a graded family of monomial ideals $\I$ such that $\Delta(\I) = P$. Particularly, one can take $\I = \{I_k\}_{k \in \NN}$ with
	$$I_k = \left\langle \left\{x^\a ~\big|~ \a \in kP \cap \ZZ^n\right\}\right\rangle.$$
	
	It is not hard to show further that $\I = \overline{\I}$, and if $\J = \{J_k\}_{k \in \NN}$ is any other graded family of monomial ideals such that $\Delta(\J) = P$ then we have $\J \subseteq \I$, i.e., $J_ k \subseteq I_k$ for all $k \in \NN$.
\end{example}

The next example illustrates that even if $\Delta(\I)$ is a rational polyhedron, the conditions in Theorem \ref{lem.NPpoly} may still not be satisfied.

\begin{example} \label{ex.nonpolyP}
	Let $R = \kk[x,y]$ and consider the graded family $\I = \{I_k\}_{k \in \NN}$ with
	$$I_k = (x,y)^{\lceil k/2 \rceil + 1} \subseteq R.$$
	It can be seen that the Newton polyhedron of $(x,y)^k$ is $\conv\{(k,0), (0,k)\} + \RR^2_{\ge 0}$. It follows that, for $k \in \NN$,
	$$\dfrac{1}{k}\NP(I_k) = \conv \left\{\left(\dfrac{\lceil k/2\rceil+1}{k},0\right), \left(0, \dfrac{\lceil k/2\rceil +1}{k}\right)\right\} + \RR^2_{\ge 0}.$$
	This, since
	$\lim_{n \rightarrow \infty} \dfrac{\lceil n/2\rceil + 1}{n} = \dfrac{1}{2},$
	implies that $\C(\I) = \bigcup_{k \in \NN}\dfrac{1}{k}\NP(I_k)$ is exactly the interior of $\conv \left\{ \left(\frac{1}{2},0\right), \left( 0, \frac{1}{2}\right)\right\} + \RR^2_{\ge 0}$ together with two open rays $\left( \frac{1}{2}, \infty\right) \times 0$ and $0 \times \left(\frac{1}{2},\infty\right).$ It is clear that $\C(\I)$ is not a polyhedron.
	
	On the other hand,
	$$\Delta(\I) = \overline{\C(\I)} = \conv \left\{ \left(\dfrac{1}{2}, 0\right), \left(0, \dfrac{1}{2}\right)\right\} + \RR^2_{\ge 0}$$
	is a rational polyhedron.
\end{example}

We are now ready to state our first main result, which establishes the equivalence between conditions in Theorem \ref{lem.NPpoly} and the Noetherian property of the Rees algebra of $\I$.

 \begin{theorem}
 	\label{thm.Noetherian}
 	Let $\I = \{I_k\}_{k \in \NN}$ be a graded family of monomial ideals in $R$ and let $\overline{\I} = \{\overline{I_k}\}_{k \in \NN}$. The following are equivalent:
 	\begin{enumerate}
 		\item $\I$ satisfies any of the conditions in Theorem \ref{lem.NPpoly}.
 		\item There exists an integer $c$ such that $\overline{I_c^k} = \overline{I_{kc}}$ for all $k \in \NN$.
 		\item $\R(\overline{\I})$ is Noetherian.
 		\item $\R(\I)$ is Noetherian.
 	\end{enumerate}
 \end{theorem}

 \begin{proof} We start by showing that $(1) \Leftrightarrow (2)$. Indeed, observe that, for any $k, c \in \NN$, $\NP(I_c^k) = k\NP(I_c)$. Thus, the condition in Theorem \ref{lem.NPpoly} is equivalent to the condition that
 	$$\NP(I_c^k) = \NP(I_{kc}).$$
 Furthermore, since the Newton polyhedron determines the integral closure of a monomial ideal. This later condition is equivalent to the fact that $\overline{I_c^k} = \overline{I_{kc}}$. That is, $(1) \Leftrightarrow (2)$.

 We continue by showing that $(2) \Leftrightarrow (3)$. Suppose first that there exists an integer $c$ such that $\overline{I_c^k} = \overline{I_{kc}}$ for all $k \in \NN$. This implies that the $c$-th Veronese subalgebra of the Rees algebra of the graded family $\overline{\I}$ is given by
 $$\R^{[c]}(\overline{\I}) = \bigoplus_{k \in \NN}\overline{I_{kc}}t^k = \bigoplus_{k \in \NN}\overline{I_c^k}t^k = \overline{\R(I_c)},$$
 the integral closure of the Rees algebra of $I_c$. It follows that $\R^{[c]}(\overline{\I})$ is Noetherian. As a consequence, by \cite[Theorem 2.1]{HHT2007}, $\R(\overline{\I})$ is also Noetherian. Thus, $(2) \Rightarrow (3)$.

 Conversely, suppose that $\R(\overline{\I})$ is Noetherian. Then, there exists an integer $c$ such that its $c$-th Veronese subalgebra is standard graded. That is, we have
 $$\overline{I_{kc}} = \overline{I_c}^k \text{ for all } k \in \NN.$$
 Therefore, for all $k \in \NN$,
 $$\overline{I_{kc}} = \overline{I_c}^k \subseteq \overline{I_c^k} \subseteq \overline{I_{kc}}.$$
 This forces $\overline{I_c^k} = \overline{I_{kc}}$ for all $k \in \NN$. This gives $(3) \Rightarrow (2)$.

 Finally, we show that $(3) \Leftrightarrow (4)$. Suppose that $\R(\overline{\I})$ is Noetherian. As above, there exists an integer $c$ such that $\overline{I_c^k} = \overline{I_{kc}}$ for all $k \in \NN$, and $S = \R^{[c]}(\overline{\I})$ is a Noetherian ring. Observe that $\R(\overline{\I})$ is integral over $S$, so $\R(\overline{\I})$ is a finitely generated $S$-module.
 In fact, to see that $\R(\overline{\I})$ is integral over $S$, take any $f\in \overline{I_k}$, we have
 $$f^c \in (\overline{I_k})^c \subseteq \overline{(I_k)^c} \subseteq \overline{I_{ck}} = (\overline{I_c})^k,$$
 hence, $f$ satisfies the equation $x^c - f^c=0$ where $f^c \in (\overline{I_c})^k \subseteq S$.
 Observe further that $S$ is the integral closure of $\R(I_c)$ in $R[t]$, so $S$ is a finitely generated module over the Rees algebra $\R(I_c)$ of $I_c$. Therefore, $\R(\overline{\I})$ is a finitely generated and, hence, Noetherian module over $\R(I_c)$. As a consequence, its submodule $\R(\I)$ is also finitely generated over $\R(I_c)$. By Eakin-Nagata theorem (see, for example, \cite[page 263]{Mat1980}), it follows that $\R(\I)$ is a Noetherian ring.


 Conversely, suppose that $\R(\I)$ is Noetherian. Then, again by \cite[Theorem 2.1]{HHT2007}, there exists an integer $c$ such that its $c$-th Veronese subalgebra is standard graded. That is, $I_{kc} = I_c^k$ for all $k \in \NN$. Particularly, we get $\overline{I_{kc}} = \overline{I_c^k}$ for all $k \in \NN$. Therefore, $(4) \Rightarrow (2) \Leftrightarrow (3)$. The theorem is proved.
 \end{proof}

Following \cite{D+2021}, for a monomial ideal $I$ and a real number $r \in \RR_{\ge 0}$, we define the \emph{$r$-th real power} of $I$ to be
$$\overline{I^r} := \{x^\a ~\big|~ \a \in r\NP(I) \cap \NN^n\}.$$

\begin{example}
	\label{ex.Noetherian}
	Let $I$ be a monomial ideal and let $f: \NN \rightarrow \RR_{\ge 0}$ be a \emph{subadditive} function; that is, $f(m) + f(n) \ge f(m+n)$ for all $m,n \in \NN$. Suppose that $\lim_{k \rightarrow \infty} \frac{f(k)}{k} \in \QQ$ and is attained at some value $k_0$. Let $\I = \{I_k\}_{k \in \NN}$ be the family of monomial ideals given by
	$$I_k = \overline{I^{f(k)}}.$$
	
	It can be seen that, since $f$ is subadditive, $\I$ is a graded family of ideals. Observe also that, again since $f$ is subadditive, $\lim_{k \rightarrow \infty} \frac{f(k)}{k} = \inf_{k \in \NN} \frac{f(k)}{k}$. It follows that, for all $k \in \NN$, we have
	$$\dfrac{1}{k}\NP\left(\overline{I^{f(k)}}\right) = \dfrac{f(k)}{k}\NP(I) \subseteq \frac{f(k_0)}{k_0}\NP(I) = \dfrac{1}{k_0}\NP\left(\overline{I^{f(k_0)}}\right).$$
	This implies that
	$$\C(\I) = \bigcup_{k \in \NN}\dfrac{1}{k}\NP(I_k) = \dfrac{f(k_0)}{k_0}\NP(I)$$
is a rational polyhedron. By Theorem \ref{thm.Noetherian}, we deduce that $\R(\I)$ is a Noetherian algebra. This is a nontrivial fact for many choices of the function $f$; for instance if, for some $r \in \QQ$,
	$$f(n) = \dfrac{rn-1+\sqrt{(rn)^2 + 2\lceil rn \rceil + 1}}{rn}.$$
\end{example}

\begin{example}
	\label{ex.nonNoeth}
	Let $\I = \{I_k\}_{k \in \NN}$, where $I_k = (x,y)^{\lceil k/2 \rceil + 1} \subseteq \kk[x,y]$, be the graded family as in Example \ref{ex.nonpolyP}. As shown in Example \ref{ex.nonpolyP}, the limiting body $\C(\I)$ is not a polyhedron. Thus, by Theorem \ref{thm.Noetherian}, we know that the Rees algebra $\R(\I)$ is not Noetherian. This is a nontrivial fact otherwise.
\end{example}

\begin{remark}
	In \cite{C+2021}, the \emph{asymptotic Newton polyhedron} associated to a graded family $\I$ of monomial ideals is defined to be the limiting body $\C(\I)$ if $\C(\I)$ is a polyhedron. Theorems \ref{lem.NPpoly} and \ref{thm.Noetherian} basically say that the asymptotic Newton polyhedron of $\I$ exists precisely when the Rees algebra $\R(\I)$ is Noetherian.
\end{remark}

As an immediate consequence of Theorem \ref{thm.Noetherian}, we obtain the following corollary, whose last assertion recovers \cite[Corollary 2.2]{HHT2007}.

\begin{corollary}
	\label{cor.NoetherianR}
	Let $I$ be a monomial ideal with a monomial decomposition $I = J_1 \cap \dots \cap J_s$. Let $\I = \{I_k\}_{k \in \NN}$ be a graded family where $I_k = J_1^k \cap \dots \cap J_s^k$. Then the Rees algebra $\R(\I)$ is Noetherian, and there exists an integer $c$ such that
	$$\left(\bigcap_{i=1}^s J_i^c\right)^k = J_1^{kc} \cap \dots \cap J_s^{kc} \text{ for all } k \in \NN.$$
\end{corollary}

\begin{proof} By \cite[Theorem 3.11]{C+2021}, we have that
	$$\bigcup_{k \in \NN} \dfrac{1}{k}\NP(I_k) = \NP(J_1 ) \cap \dots \cap \NP(J_s)$$
	is a polyhedron. Thus, it follows from Theorem \ref{thm.Noetherian} that $\R(\I)$ is Noetherian. This establishes the first assertion.
	
	Since $\R(\I)$ is Noetherian, there exists an integer $c$ such that its $c$-th Veronese subalgebra is standard graded. This proves the later assertion.
\end{proof}

Theorem \ref{thm.Noetherian} suggests that there is a close connection between the Newton-Okounkov bodies associated to $\I$ and $\overline{\I}$. The next proposition shows that it is indeed the case.

\begin{proposition}
	\label{lem.int}
	Let $\I = \{I_k\}_{k \in \NN}$ be a graded family of monomial ideals in $R$. Set $\overline{\I} = \{\overline{I_k}\}_{k \in \NN}.$ Then, $\overline{\I}$ is also a graded family of ideals in $R$, and we have
	$$\Delta(\I) = \Delta(\overline{\I}).$$
\end{proposition}

\begin{proof} We shall first show that $\overline{\I}$ is a graded family of ideals in $R$. Consider arbitrary $p, q \in \NN$, and monomials $f \in \overline{I_p}$ and $g \in \overline{I_q}$. Since we are working with monomial ideals, this implies that there exist positive integers $r$ and $s$ such that
	$$f^r \in I_p^r \text{ and } g^s \in I_q^s.$$
	Particularly, we get
	$$(fg)^{rs} = f^{rs}g^{rs} \in I_p^{rs}I_q^{rs} = (I_pI_q)^{rs} \subseteq (I_{p+q})^{rs}.$$
	Thus, by definition, $fg \in \overline{I_{p+q}}$. This shows that $\overline{I_p}\cdot \overline{I_q} \subseteq \overline{I_{p+q}}$, and $\overline{\I}$ is a graded family of ideals.
	
	To prove the equality $\Delta(\I) = \Delta(\overline{\I})$, observe first the inclusion $\Delta(\I) \subseteq \Delta(\overline{\I})$ is trivia. We shall prove the other inclusion. Consider any $\frac{\a}{k} \in \Delta(\overline{\I})$, for which $x^\a \in \overline{I_k}$. Since $I_k$ is a monomial ideal, there exists a positive integer $r$ such that $x^{r\a} = (x^\a)^r \in I_k^r \subseteq I_{kr}$. This implies that
	$$\dfrac{\a}{k} = \dfrac{r\a}{kr} \in \Delta(\I).$$
	Hence, $\Delta(\overline{\I}) \subseteq \Delta(\I)$, and the result is proved.
\end{proof}

\begin{remark}
	\label{rmk.direct}
	Inspired by Proposition \ref{lem.int}, one may ask if there is a direct proof for $(1) \Leftrightarrow (4)$ without going through $(3)$ in Theorem \ref{thm.Noetherian}.
\end{remark}


\section{Newton-Okounkov body and the analytic spread} \label{sec.AS}

In this section, we present a formula that interprets the analytic spread of a graded family $\I$ of monomial ideals in terms of its Newton-Okounkov body when $\I$ is Noetherian. Recall that for a polyhedron $P$, $\mdc(P)$ denotes the maximum dimension of a compact face of $P$.

Our next main result reads as follows.

 \begin{theorem}
 	\label{thm.analyticSpread}
 	Let $\I$ be a Noetherian graded family of monomial ideals in $R$. Let $\Delta(\I)$ and $\ell(\I)$ be its Newton-Okounkov body and analytic spread, respectively. Then,
 	$$\ell(\I) = \mdc(\Delta(\I))+1.$$
 \end{theorem}

\begin{proof} By \cite[Theorem 2.1]{HHT2007}, since $\I$ is Noetherian, there exists an integer $c$ such that the $c$-th Veronese subalgebra $S = \R^{[c]}(\I)$ of $\R(\I)$ is standard graded. Particularly, we get
	\begin{align}
		I_c^k = I_{kc} \text{ for all } k \in \NN, \label{eq.Ikc}
	\end{align}
	and so, $S = \R(I_c)$ is the Rees algebra of the ideal $I_c$.
	Furthermore, it is easy to see that $\R(\I)$ is integral over $S$, so $\R(\I)$ is a finitely generated $S$-module. Thus,
	$$\ell(\I) = \dim (\R(\I)/\mm \R(\I)) = \dim S/\mm S = \ell(I_c).$$
	
	Observe that (\ref{eq.Ikc}) implies that
	$$\dfrac{1}{c}\NP(I_c) = \dfrac{1}{kc}\NP(I_c^k) = \dfrac{1}{kc}\NP(I_{kc}) \text{ for all } k \in \NN.$$
	By Theorem \ref{lem.NPpoly}, we then get
	$$\Delta(\I) = \dfrac{1}{c} \NP(I_c).$$
	It follows that $\mdc(\Delta(\I)) = \mdc(\NP(I_c))$. The conclusion now follows from \cite[Theorem 2.3]{BA2003}, which established the equality $\ell(I_c) = \mdc(\NP(I_c))+1$.
\end{proof}

\begin{example}
	\label{ex.AS1}
Let $\I$ be a graded family of ideals given in Example \ref{ex.Noetherian}. Then, as observed, $\I$ is Noetherian. Theorem \ref{thm.analyticSpread} now gives $\ell(\I) = \mdc(\Delta(\I)) + 1 = \mdc\left(\NP(I)\right) + 1 = \ell(I)$.
\end{example}

The following example illustrates that when $\I$ is not Noetherian, even if $\Delta(\I)$ is a polyhedron, the formula for analytic spread in Theorem \ref{thm.analyticSpread} does not necessarily hold.

\begin{example} \label{ex.ASnot=mdc}
	Consider the graded family $\mathcal{I}$ given by $I_0 = R$ and $I_n=(a,b)^4(x,y)^n$, for $n \ge 1$. The special fiber ring of $\I$ is
	$$\mathcal{F}(\mathcal{I})=  \frac{R}{\mm} \otimes_R \mathcal{R}(\mathcal{I})= \frac{R}{\mm} \bigoplus \frac{(a,b)^4(x,y)}{\mm(a,b)^4(x,y)} \bigoplus \ldots \bigoplus \frac{(a,b)^4(x,y)^n}{\mm(a,b)^4(x,y)^{n}} \bigoplus \ldots .$$
	
	For every $m,n\ge 1$, if $f \in \mathcal{F}(\mathcal{I})_n = \frac{(a,b)^4(x,y)^n}{\mm(a,b)^4(x,y)^{n}}$ and $g\in \mathcal{F}(\mathcal{I})_n$, since the representative of $fg$ is in $(a,b)^8(x,y)^{m+n} \subseteq \mm(a,b)^4(x,y)^{m+n}$, we have $fg =0$ in $\mathcal{F}(\mathcal{I})$. Thus, for any prime ideal $\pp$ in $\mathcal{F}(\mathcal{I})$, $\pp$ must contain all positively graded elements of $\mathcal{F}(\mathcal{I})$, i.e,
	$$\bigoplus_{n\ge 1}\mathcal{F}(\mathcal{I})_n \subseteq \pp.$$
	On the other hand, $\pp \cap \mathcal{F}(\mathcal{I})_0 = \pp \cap \kk$ is a prime ideal in $\kk$, thus $\pp \cap \kk =0$. Therefore, the only prime ideal in $\mathcal{F}(\mathcal{I})$ is $\bigoplus_{n\ge 1}\mathcal{F}(\mathcal{I})_n$, hence, $\ell(\mathcal{I})=0$.
	
	Note also that for each $n$,
	\[
	\frac{1}{n}\NP(I_n) = \conv \left\{\left(\frac{4}{n},0,1,0\right),\left(\frac{4}{n},0,0,1\right),\left(0,\frac{4}{n},1,0\right),\left(0,\frac{4}{n},0,1\right) \right\} + \RR^n_{\ge 0},
	\]
	and, furthermore, the defining hyperplanes of $\frac{1}{n}\NP(I_n)$ are precisely
	\[
	a\ge 0, b\ge 0, x\ge 0, y\ge 0, a+b\ge \frac{4}{n}, x+y\ge 1.
	\]
	Therefore,
	\[
	\Delta(\mathcal{I})=\overline{\RR^n_{\ge 0} \cap (x+y \ge 1) \cap \left(\bigcup_{n\in \NN} (a+b\ge \frac{4}{n}) \right)} = \RR^n_{\ge 0} \cap (x+y \ge 1),
	\]
	which is a polyhedron with two vertices $(0,0,1,0)$ and $(0,0,0,1)$. Hence, $\mdc(\Delta(\mathcal{I}))=1$, and $\ell(\mathcal{I})\not=\mdc(\Delta(\mathcal{I}))+1$.
\end{example}

We immediately obtain the following result as a consequence of Theorem \ref{thm.analyticSpread}. For a graded family $\I = \{I_k\}_{k \in \NN}$ of monomial ideals, set $\overline{\I} = \{\overline{I_k}\}_{k \in \NN}$. As shown in Proposition \ref{lem.int}, $\overline{\I}$ is also a graded family of monomial ideals.

\begin{corollary}
	\label{cor.AS-int}
	For any Noetherian graded family $\I$ of monomial ideals, we have
	$$\ell(\I) = \ell(\overline{\I}).$$
	Particularly, if $\J$ is any Noetherian graded family of monomial ideals such that $\overline{\I} = \overline{\J}$ then $\ell(\J) = \ell(\I).$
\end{corollary}

\begin{proof} By Proposition \ref{lem.int}, we have $\Delta(\I) = \Delta(\overline{\I})$. The first assertion then follows from Theorems \ref{thm.analyticSpread} and \ref{thm.Noetherian}. The second assertion is a direct consequence of the first one.
\end{proof}


The following corollary to Theorem \ref{thm.analyticSpread} gives a generalization of the combinatorial interpretation for the analytic spread of a monomial ideal, which was due to Bivi\`a-Ausina \cite[Theorem 2.3]{BA2003} (see also \cite[Corollary 4.10]{Sin2007}), to that for the \emph{symbolic} analytic spread.

\begin{corollary}
	\label{cor.ellsG}
	Let $I \subseteq R$ be a monomial ideal. Then,
	$$\ell_s(I) = \mdc(\SP(I))+1.$$
\end{corollary}

\begin{proof} The assertion follows directly from Theorem \ref{thm.analyticSpread} and Remark \ref{rmk.NPSP}.
\end{proof}

\begin{example}
	\label{ex.symbolicAS}
Let $I = (xy, yz, zx) = (x,y) \cap (y,z) \cap (z,x) \subseteq R = \kk[x,y,z]$.
By the description of $\SP(I)$ given in (\ref{eq.SP}), the defining half-spaces for $\SP(I)$ are: $a_1 \ge 0, a_2 \ge 0, a_3 \ge 0, a_1+a_2 \ge 1, a_2+a_3 \ge 1$ and $a_3+a_1 \ge 1$. Thus,
$$\SP(I) = \conv\left\{(1,1,0), (1,0,1), (0,1,1), \left(\frac{1}{2}, \frac{1}{2}, \frac{1}{2}\right)\right\} + \RR^3_{\ge 0}.$$

It is easy to see that the compact faces of maximum dimension of $\SP(I)$ are the lines connecting $\left(\frac{1}{2}, \frac{1}{2}, \frac{1}{2}\right)$ and any of the other three vertices. These faces all have dimension 1. Corollary \ref{cor.ellsG} now gives $\ell_s(I) = 2$.
\end{example}

\begin{remark}
    \label{rmk.SymbAnaSpread1}
For $I = \cap_{\pp \in \Min(I)} \pp^{\omega_\pp}$, we have that $\ell_s(I)\ge 1$ and $\ell_s(I)=1$ if and only if $\SP(I)$ has only one vertices (otherwise some segment through $2$ vertices gives us a compact face of dimension $1$), if and only if $I$ is a principal ideal. In this case, $I^{(k)}=I^k$ for all $k$.
\end{remark}

\begin{corollary}
    \label{cor.SymbAnaLessN}
Let $n\ge 2$ and $I = \cap_{\pp \in \Min(I)} \pp^{\omega_\pp}$ be a monomial ideal in $S$ such that $\mm \not \in \Ass(I)$. Then $$\ell_s(I) \le n-1=\dim(S)-1.$$
In particular, $\ell_s(I)=\dim(S)$ if and only if $\mm\in \Ass(I)$.
\end{corollary}
\begin{proof}
Since $I\not = \mm$, all defining hyperplanes of $\SP(I)$ give us noncompact faces. Hence, there is no compact face that has maximum dimension $n-1$. It follows that obvious that $\ell_s(I)= \mdc(\SP(I)) +1 \le n-2 +1=n-1$. Moreover, $\ell_s(I)=\dim(S)=n$ if and only if $SP(I)$ has a compact face of dimension $n-1$ if and only if $\mm \in \Ass(I)$.
\end{proof}

When $\R(\I)$ is not necessarily Noetherian, the proof of Theorem \ref{thm.analyticSpread} no longer works. It was shown recently by Cutkosky and Sarkar in \cite[Lemma 3.6]{CS2021} that if $\I$ is a filtration of ideals in a local ring $R$ then $\ell(\I)$ exists and is bounded above by the dimension of $R$. Their arguments in fact work for any graded family of ideals in a polynomial ring, as seen below. That is, the invariant $\ell(\I)$ in Definition \ref{def.AS} is always finite.

Recall from \cite[Definition 3.2]{CS2021} that, for a graded family $\I = \{I_k\}_{k \in \NN}$ of homogeneous ideals in $R$ and a positive integer $a$, the \emph{$a$-th truncation} of $\I$ is the graded family $\I_a = \{I_{a,k}\}_{k \in \NN}$, where
$$I_{a,k} = \left\{\begin{array}{lll} I_k & \text{if} & k \le a \\ \sum_{\substack{i,j > 0 \\ i+j=k}} I_{a,i}I_{a,j} & \text{if} & n > a. \end{array}\right.$$
Clearly, by the definition, the Rees algebra $\R(\I_a)$ is finitely generated for any $a \in \NN$. It follows by, for instance, \cite[Theorem 2.1]{HHT2007} that there exists an integer $e_a \in \NN$ such that the $e_a$-th Veronese subalgebra $S_a = \R^{[e_a]}(\I_a)$ of $\R(\I_a)$ is standard graded and $\R(\I_a)$ is a finite module over $S_a$. We shall define a similar but slightly different notion of truncation.

\begin{definition}
	\label{def.truncation}
	Let $\I = \{I_k\}_{k  \in \NN}$ be a graded family of homogeneous ideals in $R$ and let $a \in \NN$. Define the \emph{$a$-th upper truncation} of $\I$ to be the graded family $\I_a^* = \{I_{a,k}^*\}_{k \in \NN}$, where
	$$I_{a,k}^* = \left\{\begin{array}{lll} I_k & \text{if} & k \le a \\ \left(I_{e_a}\right)^r & \text{if} & k = re_a, r \in \NN \\ \sum_{\substack{i,j > 0 \\ i+j=k}} I_{a,i}^* I_{a,j}^* & \text{if} & k \text{ is otherwise}. \end{array}\right.$$
\end{definition}

\begin{lemma}
	\label{lem.gradedI*}
	The family $\I^*_a$ constructed in Definition \ref{def.truncation} is a graded family of ideals.
\end{lemma}

\begin{proof} It suffices to show that for any positive integers $k$ and $l$ such that $k+l$ is a multiple of $e_a$, we have
	\begin{align}
		I^*_{a,k}I^*_{a,l} \subseteq I^*_{a,k+l}. \label{eq.gradedI*}
	\end{align}
	Clearly, $I^*_{a,k}$ and $I^*_{a,l}$ are generated by $I_1, \dots, I_a$ and $I_{e_a}$. Thus, it is enough to prove (\ref{eq.gradedI*}) when $k$ and $l$ are both at most $e_a$. Furthermore, in this case, (\ref{eq.gradedI*}) holds because $\I$ is a graded family.
\end{proof}

Again, it is straightforward from the definition that $\R(\I_a^*)$ is finitely generated and $\R^{[e_a]}(\I_a^*) = \R(I_{e_a})$ is a standard graded algebra.

We shall make use of the following technical lemma from \cite{CS2021}.

\begin{lemma}
	\label{lem.CS2021}
	Let $\A$ be a $\NN$-graded ring. Suppose that $\{\A_a\}_{a \in \NN}$ is a collection of Noetherian graded rings with the same grading as $\A$, $\max\{\dim \A_a ~\big|~ a \in \NN\} < \infty, \A_{a,k} = \A_k$ for all $k \ge a$ and all $a \in \NN$, and there is a graded homomorphism $\phi_a : \A_a \rightarrow \A$ such that $\phi_a(x) = x$ for all homogeneous elements of $\A_a$ of degree at most $a$. Then,
	$$\dim \A \le \max \{\dim \A_a ~\big|~ a \in \NN\}.$$
\end{lemma}

By applying Lemma \ref{lem.CS2021} in essentially the same way as in \cite{CS2021}, we obtain a slight improvement of \cite[Lemma 3.6]{CS2021}.

\begin{proposition}
	\label{lem.CS3.6}
	Let $\I$ be a graded family of homogeneous ideals in $R$ and let $\I_a^*$ be its $a$-th upper truncation, for $a \in \NN$. Then,
	$$\ell(\I) \le \max\{\ell(\I_a^*) ~\big|~ a \in \NN\} = \max\{\ell(I_{e_a}) ~\big|~ a \in \NN\}.$$
	In particular, we have $\ell(\I) \le \dim R$.
\end{proposition}

\begin{proof} Let $\A = \R(\I)/\mm\R(\I)$ and $\A_a = \R(\I_a^*)/\mm\R(\I_a^*)$. The homomorphism $\phi_a: \A_a \rightarrow \A$, defined by $\phi_a(x+\mm I_{a,k}^*) = x + \mm I_k$ for homogeneous elements $x \in I_{a,k}^*/\mm I_{a,k}^*$, is a graded ring homomorphism that satisfies the condition of Lemma \ref{lem.CS2021}. Thus, by Lemma \ref{lem.CS2021}, we have
	$$\ell(\I) = \dim \A \le \max\{\dim \A_a ~\big|~ a \in \NN\} = \max \{\ell(\I_a^*) ~\big|~ a \in \NN\}.$$
	Furthermore, since $\R(\I_a^*)$ is a finite module over $\R^{[e_a]}(\I_a^*)$, it follows that $\R(\I_a^*)/\mm \R(\I_a^*)$ is an integral extension of $\R(I_{e_a})/\mm \R(I_{e_a})$. Therefore, $\ell(\I_a^*) = \ell(I_{e_a})$, and the first assertion is proved.
	
	The second assertion follows since, for any $a \in \NN$, $\ell(I_{e_a}) \le \dim R$ is a well known fact.
\end{proof}

\begin{example}
		Consider the graded family $\I$ as in Example \ref{ex.ASnot=mdc}. Since for each $n$, $\NP(I_n)$ has four vertices as above discussion, one can see that $\mdc(\NP(I_n))=2$, hence $\ell(I_n)=3$ for every $n\in \NN$. Thus, $\ell(\mathcal{I})<\ell(I_n)$ for all $n \in \NN$. That is, the first inequality in Proposition \ref{lem.CS3.6} may be a strict inequality.
\end{example}

\begin{remark}
	\label{rmk.ellNoetherian}
	It is easily seen that if $\R(\I)$ is Noetherian then so is $\R(I)/\mm\R(\I)$. Thus, in this case, the Hilbert function of the graded $\kk$-algebra $\R(\I)/\mm\R(\I)$ is eventually a periodic polynomial of degree equal to $\dim \R(\I)/\mm\R(\I) - 1 = \ell(\I)-1$. Furthermore, $\R(\I)/\mm\R(\I) = \R(\I) \otimes_R R/\mm$. Therefore, the Hilbert function of $\R(\I)/\mm\R(\I)$ at degree $k$ measures exactly the minimum number of generators for $I_k$. Particularly, we get
	$$\ell(\I) = \min\{t \in \RR ~\big|~ \mu(I_k) = O(k^{t-1})\}.$$
\end{remark}

When the Rees algebra $\R(\I)$ is not known to be Noetherian, we can still use the right hand side of the last equality to define a new invariant associated to the graded family $\I$; though it is not known in general when this new invariant is the same as $\ell(\I)$.

\begin{definition}
	\label{def.analyticSpread*}
	Let $\I$ be a graded family of ideals in $R$. Define
	$$\ell^*(\I) = \min\{t \in \RR ~\big|~ \mu(I_k) = O(k^{t-1})\}.$$
\end{definition}


It is not clear if for any given graded family $\I$ of ideals (or monomial ideals) in $R$, $\ell^*(\I)$ is finite.
The following example shows that $\ell(\I) \not= \ell^*(\I)$ in general.

\begin{example}
\label{ex.ell*notell}
    Consider the graded family $\I$ as in Example \ref{ex.ASnot=mdc}. By a simple count, we have $\mu(I_n)=5(n+1)$ for every $n$, hence $\ell^*(\I)=2 \ne \ell(\I)$.
\end{example}

\begin{question}
	\label{question.ell}
	For what graded families $\I$ of monomial ideals that we have $\ell(\I) = \ell^*(\I)$?
\end{question}

Inspired by Proposition \ref{lem.CS3.6} and in connection to Question \ref{question.ell}, we raise the following question.

\begin{question}
	\label{quest.ell*}
	Let $\I$ be a graded family of monomial ideals. Is $\ell^*(\I) \le \dim R$?
\end{question}

Example \ref{ex.ASnot=mdc} provides a family that fails $\ell(\mathcal{I})\not=\mdc(\Delta(\mathcal{I}))+1$ even when $\Delta(\I)$ is a polyhedron. Nevertheless, we saw in Example \ref{ex.ell*notell}, that $\ell^*(\mathcal{I})=2=\mdc(\Delta(\mathcal{I}))+1$. This motivates the question of when $\ell^*(\mathcal{I})=\mdc(\Delta(\mathcal{I}))+1$. We shall provide instances where this formula holds even though $\I$ is not Noetherian.

Recall that an ideal $I$ is called \emph{$\mm$-full} if $\mm I : x = I$ for some $x \in \mm$. Particularly, integrally closed monomial ideals are $\mm$-full (see, for instance, \cite[Theorem 2.4]{Goto1987}).

\begin{proposition}
	\label{prop.boundell*}
	Let $\I = \{I_k\}_{k \in \NN}$ be a graded family of $\mm$-primary, $\mm$-full monomial ideals. Suppose that $\Delta(\I)$ is a rational polyhedron. Then,
	$$\ell^*(\I) = \mdc(\Delta(\I)) + 1 = \dim R.$$
\end{proposition}

\begin{proof} Observe that, since $I_k$ is $\mm$-primary, $\RR^n_{\ge 0} \setminus \frac{1}{k}\NP(I_k)$ is a bounded set for all $k \in \NN$. This implies that $\RR^n_{\ge 0} \setminus \Delta(\I)$ is a bounded set. Thus, $\mdc(\Delta(\I)) = n-1$, i.e., $\mdc(\Delta(\I)) + 1 = \dim R$.
	
Define $\I^* = \{I_k^*\}_{k \in \NN}$ to be the graded family, as in Example \ref{ex.nonpolyDelta}, given by
	$$I_k^* = \langle \{x^\a ~\big|~ \a \in k\Delta(\I) \cap \ZZ^n\}\rangle.$$
	As shown in \cite[Proposition 2]{Wol2008}, $\I^*$ is a graded family of monomial ideals and
	$$\bigcup_{k \in \NN} \frac{1}{k}\NP(I_k^*) = \Delta(\I^*) = \Delta(\I).$$
	Particularly, this implies that the Rees algebra $\R(\I^*)$ is Noetherian, by Theorem \ref{thm.Noetherian}. Therefore, by Theorem \ref{thm.analyticSpread}, we have
	$$\ell^*(\I^*) = \ell(\I^*) = \mdc(\Delta(\I^*))+1 = n.$$
	Obviously, we also have $I_k \subseteq I_k^*$ for all $k \in \NN$. Since $I_k$ and $I_k^*$ are both $\mm$-primary, $\lambda(I_k^*/I_k) < \infty$. Furthermore, $I_k$ is $\mm$-full, so it follows from \cite[Lemma 2.2]{Goto1987} that $\mu(I_k^*) \le \mu(I_k)$ for all $k \in \NN$. Hence, $\ell^*(\I^*) \le \ell^*(\I)$, and we have $\ell^*(\I) \ge n.$ \par
    \vspace{0.5em}
	
	On the other hand, since $I_1$ is $\mm$-primary, there exists an integer $d$ such that $\mm^d \subseteq I_1$. It follows that $\mm^{dk} \subseteq (I_1)^k \subseteq I_k$ for all $k \in \NN$. Observe that $\mm^{dk}$ is integrally closed, so it is $\mm$-full. Clearly, $\lambda(I_k/\mm^{dk}) < \infty$. Therefore, again by \cite[Lemma 2.2]{Goto1987}, we then have $\mu(I_k) \le \mu(\mm^{dk})$ for all $k \in \NN$. Moreover, it can be seen that $\mu(\mm^{dk}) = O(k^{n-1})$. Hence, $\ell^*(\I) \le n$. The result is established.
\end{proof}



We continue by considering the invariant $\ell^*(\I)$ for a slightly more general case, where the Newton-Okounkov body is defined with respect to a \emph{good} valuation. We refer the interested reader to \cite{Cut2013, KK2014} for the definition and existence of good valuations, when $R$ is a regular local ring.




\begin{definition}
	Let $v$ be a good valuation of $R$ and let $\I = \{I_k\}_{k \in \NN}$ be a graded family of $\mm$-primary ideals. For $k \in \NN$, set
	$$\text{val}(I_k) = \{v(f) ~\big|~ f \in I_k \setminus \{0\}\}.$$
	The \emph{limiting} and the \emph{Newton-Okounkov} bodies of $\I$, \emph{with respect to} $v$, are defined to be
	$$\C(\I) = \bigcup_{k \in \NN} \left\{\dfrac{\a}{k} ~\big|~ \a \in \text{val}(I_k)\right\} \subseteq \RR^n \text{ and } \Delta(\I) = \overline{\C(\I)} \subseteq \RR^n.$$
\end{definition}

Now, let $K$ be the quotient field of $R$. Let $v: K \setminus \{0\} \rightarrow \ZZ^n$ be a good valuation of $R$. We say that $v$ \emph{respects the monomials} of $R$, if for every monomial $x^\a = x_1^{a_1} \cdots x_n^{a_n} \in R$, $v(x^\a) = \a$. Note that in this case, the value group $S = v(R \setminus \{0\}) \cup \{0\} = \ZZ_{\ge 0}^n$. Good valuations that respect the monomials in $R$ always exist as illustrated in the following example.

\begin{example}
Fix a monomial order $>$ on $R$, and consider the Gr\"obner valuation $v: K \setminus \{0\} \rightarrow \ZZ^n$ given by
$$v(f) = \text{the multidegree of the least term of } f \text{ with respect to } >,$$ for $f\in R$, and extend to $K$ by $v(f/g)=v(f)-v(g)$. Then, $v$ is a good valuation that respects the monomials of $R$.
\end{example}

\begin{theorem}
    \label{thm.ell*valuation}
    Let $\I$ be a graded family of $\mm-$primary homogeneous ideals, and let $v$ be a good valuation that respects the monomials in $R$. Let $\C(\I)$ and $\Delta(\I)$ be the limiting and Newton-Okounkov bodies of $\I$ defined by $v$. Suppose that $\C(\I)$ is a polyhedron. Then, we have
    $$\ell^*(\I) = \mdc(\Delta(\I)) + 1 = \dim R.$$
\end{theorem}

\begin{proof}
    Note that, for each $k \in \NN$, as $I_k$ is $\mm-$primary, the sets $\ZZ_{\ge 0}^n \setminus \text{val}(I_k)$ and $\ZZ_{\ge 0}^n \setminus \text{val}(\mm I_k)$ are finite. By \cite[Proposition 7.9]{KK2014}, we have
    $$\dim_{\kk}(R/I_k) = \#(\ZZ_{\ge 0}^n \setminus \text{val}(I_k)) ,\text{ and  } \dim_{\kk}(R/\mm I_k) = \#(\ZZ_{\ge 0}^n \setminus \text{val}(\mm I_k)),$$
    where $\#(P)$ denotes the number of integral points in the set $P\subseteq \RR^n$. Therefore,
    $$\mu(I_k) = \dim_{\kk}(I_k/\mm I_k) = \#(\ZZ_{\ge 0}^n \setminus \text{val}(\mm I_k)) - \#(\ZZ_{\ge 0}^n \setminus \text{val}(I_k)).$$

    For $k \in \NN$, define
    $$\Tilde{I}_k = \langle \{x^{v(f)} ~\big|~ f \text{ is a homogeneous element of }\, I_k\setminus \{0\} \} \rangle.$$
    It can be seen that $\Tilde{I}_k$ is a monomial ideal. Moreover, as the valuation $v$ respects the monomials of $R$, we have
    $$\text{val}(I_k) = \text{val}(\Tilde{I}_k) =\mathcal{L}(\Tilde{I}_k).$$
    We will show that $$\text{val}(\mm I_k) = \text{val}(\mm \Tilde{I}_k) =\mathcal{L}(\mm \Tilde{I}_k).$$
    Indeed, since $\mm \Tilde{I}_k$ is a monomial ideal, the second equality is clear. For the first equality, observe that any element in $\mm I_k$ can be written as $\sum x_i f_i$, where $f_i\in I_k$, and $v(\sum x_i f_i) = v(x_i f_i) = v(x_i) + v(f_i)$ for some $i$. The element $x_i x^{v(f_i)} \in \mm \Tilde{I_k}$ satisfies $v(x_i x^{v(f_i)}) = v(x_i) + v(f_i)$ since $v$ respects the monomials of $R$. Thus, it follows that
    $$\text{val}(\mm I_k) \subseteq \text{val}(\mm \Tilde{I}_k).$$
    Conversely, since $\Tilde{I_k}$ is a monomial ideal, any element in $\mm \Tilde{I_k}$ can be written as $\sum x_i x^{m_i}$ where $x^{m_i}$ is a monomial in $\Tilde{I_k}$. Hence, $v(\sum x_i x^{m_i}) = v(x_i) + m_i$ for some $i$. On the other hand, since $x^{m_i} \in \Tilde{I}_k$, there exists a homogeneous element $f \in I_k\setminus \{0\}$ such that $x^{v(f)}$ divides $x^{m_i}$. Therefore, the element $x_i x^{m_i - v(f)} f \in \mm I_k$ satisfies $v(x_i x^{m_i - v(f)} f) = v(x_i)+m_i$ and this establish the desired containment
    $$\text{val}(\mm \Tilde{I}_k) \subseteq \text{val}(\mm I_k).$$

    Now, for $k\in \NN$, we have
    \begin{align*}
        \mu(I_k) &= \#(\ZZ_{\ge 0}^n \setminus \text{val}(\mm I_k)) - \#(\ZZ_{\ge 0}^n \setminus \text{val}(I_k))\\
        & = \#(\ZZ_{\ge 0}^n \setminus \text{val}(\mm I_k)) - \#(\ZZ_{\ge 0}^n \setminus \text{val}(I_k)) = \mu(\Tilde{I}_k).
    \end{align*}
    Furthermore, as $\text{val}(I_k) = \text{val}(\Tilde{I}_k) = \mathcal{L}(\Tilde{I}_k)$ for $k \in \NN$, we get $\Delta(\I) = \Delta(\widetilde{\I})$ and $\C(\I) = \C(\widetilde{\I})$, where $\widetilde{\I} = \{ \widetilde{I_k}\}_{k\ge 0}$. Since $\widetilde{\I}$ is a graded family of monomial ideals and $\C(\widetilde{\I})$ is a polyhedron, by Theorem \ref{lem.NPpoly}, $\widetilde{\I}$ is Noetherian. Hence, it follows from Theorem \ref{thm.analyticSpread} that
    $$\ell(\widetilde{\I})=\ell^*(\widetilde{\I}) = \mdc(\Delta(\widetilde{\I})) + 1.$$
    This, together with the established fact that $\mu(I_k)=\mu(\Tilde{I}_k)$ for all $k \in \NN$, implies that
    $$\ell^*(\I)=\ell^*(\widetilde{\I}) = \mdc(\Delta(\widetilde{\I})) + 1 = \mdc(\Delta(\I)) + 1.$$
    The equality $\mdc(\Delta(\I)) + 1 = \dim R$ is clear by a similar argument as that of the proof of Proposition \ref{prop.boundell*}.
\end{proof}

\begin{remark}
    In Theorem \ref{thm.ell*valuation}, if, in addtion, $\I$ is a Noetherian graded family, then $$\ell(\I) = \ell^*(\I) = \mdc(\Delta(\I)) + 1 = \dim R.$$
\end{remark}


 \section{Newton and symbolic polyhedra of monomial ideals} \label{sec.NPSP}

This section focuses on the Newton and symbolic polyhedra of monomial ideals as particular cases of the Newton-Okounkov body. We shall use our results in Section \ref{sec.NO} to investigate the relationships between $\NP(I)$ and $\SP(I)$, for a monomial ideal $I$, and their algebraic implications.

We start with a simple application of the main result in Section \ref{sec.NO}, namely, Theorem \ref{thm.Noetherian}.

\begin{theorem}
	\label{thm.NPSP}
	Let $I \subseteq R$ be a monomial ideal. There exists an integer $c$ such that
	$$\NP(I^{(c)}) = c\SP(I).$$
\end{theorem}

\begin{proof} Let $\I = \{I^{(k)}\}_{k \in \NN}$ be the family of symbolic powers of $I$. It is known that $\R(\I)$ is Noetherian, by \cite[Theorem 3.2]{HHT2007}. Thus, by Theorem \ref{thm.Noetherian}, there exists an integer $c$ such that
	$$\SP(I) = \Delta(\I) = \dfrac{1}{c}\NP(I^{(c)}).$$
	That is,
	$\NP(I^{(c)}) = c\SP(I),$
	and the statement is proved.
\end{proof}

\begin{example}
	\label{ex.thm5.1}
	Let $I = (xy,yz,zx) \subseteq \kk[x,y,z]$ be as in Example \ref{ex.symbolicAS}. It can be seen that
	$$\NP(I) = \conv \left\{(1,1,0), (1,0,1), (0,1,1)\right\} + \RR^3_{\ge 0}.$$
	It is obvious from the description of $\SP(I)$ in Example \ref{ex.symbolicAS} that $\NP(I) \not= \SP(I)$.
	
	On the other hand, we have $I^{(2)} = (x^2y^2, y^2z^2, z^2x^2, xyz)$ and so
	$$\NP(I^{(2)}) = \conv\left(\{(2,2,0), (2,0,2), (0,2,2), (1,1,1)\}\right) + \RR^3_{\ge 0} = 2\SP(I).$$
\end{example}

For a squarefree monomial ideal $I$ (or, more generally, monomial ideals of linear-power type), the equality between $\NP(I)$ and $\SP(I)$ (or equivalent, the condition that $c = 1$ in Theorem \ref{thm.NPSP}) gives a nice characterization for the equality between its symbolic powers and the integral closures of its ordinary powers.

\begin{lemma}
	\label{lem.int=symb}
	Let $I \subseteq R$ be a monomial ideal of linear-power type. Then, $\NP(I) = \SP(I)$ if and only if $\overline{I^k} = I^{(k)}$ for all $k \in \NN$.
\end{lemma}

\begin{proof} Suppose first that $\NP(I) = \SP(I)$. Since $\overline{I^k} \subseteq I^{(k)}$ is already known for a monomial ideal of linear-power type, we shall prove the reverse inclusion.
	Consider any monomial $x^\a \in I^{(k)}$. Then, by Lemma \ref{lem.membership} and Remark \ref{rmk.SP}, we  have $\a \in k\SP(I) = k\NP(I)$, and so $x^\a \in \overline{I^k}$. Therefore, $I^{(k)} \subseteq \overline{I^k}$.
	
	Conversely, suppose that $\overline{I^k} = I^{(k)}$ for all $k \in \NN$. Since $\NP(I) \subseteq \SP(I)$ and they are both rational polyhedra, it suffices to show that $\SP(I) \cap \QQ^n = \NP(I) \cap \QQ^n$. Indeed, consider any $\a \in \SP(I) \cap \QQ^n$. Then, there exists an integer $t$ such that $\b = t\a \in t\SP(I) \cap \ZZ^n$. Thus, by Lemma \ref{lem.membership} and Remark \ref{rmk.SP}, we get that $x^{\b} \in I^{(t)} = \overline{I^t}$, and so $\b \in \NP(I^t) = t\NP(I)$. Particularly, it follows that $\a \in \NP(I)$, and the proof completes.
\end{proof}

Observe that, for a squarefree monomial ideal $I$, we have $I^k \subseteq \overline{I^k} \subseteq I^{(k)}$ for all $k \in \NN$. Thus, Lemma \ref{lem.int=symb} quickly recovers and strengthens \cite[Theorem 4.1]{B+2021}. Example \ref{ex.SPmonomial} shows that the conclusion of Lemma \ref{lem.int=symb} does not necessarily hold if $I$ not of linear-power type.


The constant $c$ in Theorem \ref{thm.NPSP} is given implicitly by the proof of Theorem \ref{thm.Noetherian} (and also of Theorem \ref{lem.NPpoly}). When $I$ is a squarefree monomial ideal (or, more generally, of linear-power type), we can give an explicit understanding of $c$ in terms of the vertices of the symbolic polyhedron of $I$.

\begin{theorem}
	\label{thm.c}
	Let $I \subseteq R$ be a monomial ideal of linear-power type and suppose that $\{v_1, \dots, v_r\}$ are the vertices of $\SP(I)$. Let $c$ be the least common multiple of the denominators appearing in the coordinates of $v_1, \dots, v_r$.
	\begin{enumerate}
		\item $\NP(I^{(c)}) = c\SP(I)$; equivalently, we have $I^{(kc)} = \overline{\left(I^{(c)}\right)^k}$ for all $k \in \NN$.
		\item $c$ is the smallest possible integer for which the conclusion of Theorem \ref{thm.NPSP} holds; more precisely, for any $d \in \NN$ satisfying $\NP(I^{(d)}) = d\SP(I)$, we have $c$ divides $d$.
	\end{enumerate}
\end{theorem}

\begin{proof} By definition, $\SP(I) = \conv\left( v_1, \dots, v_s\right) + \RR^n_{\ge 0}$. Thus,
	$$c\SP(I) = \conv\left(cv_1, \dots, cv_s\right) + \RR^n_{\ge 0}$$
	is an integral polyhedron. Particularly, it follows that $cv_i \in c\SP(I) \cap \ZZ^n_{\ge 0} = \mathcal{L}(I^{(c)})$. Thus, $cv_i \in \NP(I^{(c)})$ for all $i$. Hence, $c\SP(I) \subseteq \NP(I^{(c)})$. The reverse containment $c\SP(I) = \SP(I^{(c)}) \supseteq \NP(I^{(c)})$ is trivially true, and the first part of (1) is proved.
	
	To establish the second part of (1), that is, the equivalence between the equality $\NP(I^{(c)}) = c\SP(I)$ and the condition that $I^{(kc)} = \overline{\left(I^{(c)}\right)^k}$ for all $k \in \NN$, observe that, since $I$ is of linear-power type, we have $I^{(kc)} = \left(I^{(c)}\right)^{(k)}$ by \cite[Theorem 3.7]{CEHH2017}. The assertion now follows from Lemma \ref{lem.int=symb}, Lemma \ref{lem.membership} and the fact that $\SP(I^{(c)}) = c\SP(I)$.
	
	To prove (2), suppose by contrary that $c$ does not divide $d$. That is, there exists a vertex $v_i$ of $\SP(I)$ such that $dv_i \not\in \ZZ^n$. Particularly, $dv_i$ is a non-integral vertex of $d\SP(I) = \NP(I^{(d)})$. This is a contradiction since $\NP(I^{(d)})$ is an integral polyhedron. The theorem is proved.
\end{proof}

\begin{example}
	\label{ex.cFromSP1}
	Let $I = (x,y)^2 \cap (y,z)^3 \cap (z,x)^4 \subseteq \kk[x,y,z]$. Then $I$ is a monomial ideal of linear-power type. By the description of $\SP(I)$ given in (\ref{eq.SP}), the defining half-spaces of $\SP(I)$ are $a_1 \ge 0, a_2 \ge 0, a_3 \ge 0, a_1+a_2 \ge 2, a_2+a_3 \ge 3$ and $a_3+a_1 \ge 4$. Thus, the vertices of $\SP(I)$ are integral vertices (which are $(4,3,0)$, $(2,0,3)$, $(0,2,4)$) together with the vertex $\left(\frac{3}{2}, \frac{1}{2}, \frac{5}{2}\right)$. Therefore, by Theorem \ref{thm.c}, $c = 2$. This implies particularly that
	$$\NP(I^{(2)}) = 2\SP(I).$$
\end{example}

\begin{example}
	\label{ex.cFromSP2}
	Let $I = (ab,bc,cd,de,ea,fa,fb,fc,fd,fe) \subseteq \kk[a,\dots,f]$ be the edge ideal of a cone over a 5-cycle. Then,
	$$I = (a,b,d,f) \cap (a,c,d,f) \cap (a,c,e,f) \cap (b,c,e,f) \cap (b,d,e,f) \cap (a,b,c,d,e).$$
	A direct computation using (\ref{eq.SP}) gives 17 vertices of $\SP(I)$ (written as the columns of the following matrix):
	$$\left( \begin{array}{ccccccccccccccccc}
	1/5 & 1/3 & 1/2 & 0 & 0 & 0 & 1/2 & 1 & 0 & 0 & 0 & 1 & 1 & 0 & 0 & 0 & 0 \\
1/5 & 1/3 & 1/2 & 1/2& 0 & 0 & 0 & 1 & 1 & 0 & 0 & 0 & 0 & 1 & 0 & 0 & 0 \\
1/5 & 1/3 & 0 & 1/2 & 1/2 & 0 & 0 & 0 & 1 & 1 & 0 & 0 & 0 & 0 & 1 & 0 & 0 \\
1/5 & 1/3 & 0 & 0 & 1/2 & 1/2 & 0 & 0 & 0 & 1 & 1 & 0& 0 & 0 & 0 & 1 & 0 \\
1/5 & 1/3 & 0 &0 & 0 & 1/2 & 1/2 & 0 & 0 & 0 & 1 & 1 & 0 & 0 & 0 & 0 & 1 \\
2/5 & 0 & 1/2 & 1/2 & 1/2 & 1/2 & 1/2 & 0 & 0 & 0 & 0 & 0 & 1 & 1 & 1 & 1 & 1 \end{array}\right).$$
	By Theorem \ref{thm.c}, we get $c = 30$. Thus, $\NP(I^{(30)}) = 30\SP(I)$ and $c=30$ is the least possible integer for the equality in Theorem \ref{thm.NPSP} to hold.
\end{example}

 \section{Generation type and Veronese degree of monomial ideals} \label{sec.Veronese}

 In this section, we shall provide bounds for the symbolic generation type and the Veronese degree of a squarefree monomial ideal (or, more generally, monomial ideals of linear-power type) in terms of its symbolic analytic spread and the constant $c$ derived from its symbolic polyhedron as in Theorem \ref{thm.c}.

 Our first main result of the section give bounds for the symbolic Veronese degree.

 \begin{theorem}
 	\label{thm.svd}
 	Let $I$ be a monomial ideal of linear-power type and let $c$ be the constant obtained from $\SP(I)$ as in Theorem \ref{thm.c}.
 	\begin{enumerate}
 		\item If $d$ is any integer such that $I^{(dk)} = \left(I^{(d)}\right)^k$ for all $k \in \NN$ then $d$ is a multiple of $c$. Particularly, $\svd(I)$ is a multiple of $c$, and so $\svd(I) \ge c$.
 		\item $\svd(I) \le (\ell_s(I)-1)c = \mdc(\SP(I))c$.
 	\end{enumerate}
 \end{theorem}

\begin{proof} (1) Observe that powers of an ideal generated by a collection of variables are integrally closed. Moreover, the intersection of integrally closed ideals is also integrally closed. Thus, we have $\overline{I^{(dk)}} = I^{(dk)}$. It follows that $I^{(dk)} = \overline{\left(I^{(d)}\right)^k}$ for all $k \in \NN$. Since $I$ is of linear-power type, we further have $I^{(dk)} = \left(I^{(d)}\right)^{(k)}$, again by \cite[Theorem 3.7]{CEHH2017}. Therefore, by Lemma \ref{lem.int=symb}, we get that $\NP(I^{(d)}) = \SP(I^{(d)})$. On the other hand, Lemma \ref{lem.membership} implies that $\SP(I^{(d)}) = d\SP(I)$. Hence, $\NP(I^{(d)}) = d\SP(I)$. The first assertion now follows from Theorem \ref{thm.c}. The second assertion follows from the first one, by the definition of $\svd(I)$.
	
	(2) For simplicity, set $\ell = \ell_s(I) - 1$. By Corollary \ref{cor.ellsG}, Theorem \ref{thm.c} and \cite[Theorem 2.3]{BA2003}, we have $\ell = \ell(I^{(c)})-1$. Since $J = I^{(c)}$ is a monomial ideal, by \cite[Theorem 5.1]{Sin2007}, we get, for any $k \in \NN$,
	$$\left(\overline{J^\ell}\right)^k \subseteq \overline{J^{\ell k}} = J \cdot \overline{J^{\ell k - 1}} = \cdots = J^{\ell(k-1)} \cdot \overline{J^\ell} \subseteq \left(J^\ell\right)^{k-1} \cdot \overline{J^\ell} \subseteq  \left(\overline{J^\ell}\right)^k.$$
	Thus, we must have
	$\left(\overline{J^\ell}\right)^k = \overline{J^{\ell k}}$ for all $k \in \NN$. 	
	Moreover, by Theorem \ref{thm.c}, we have $I^{(c\ell k)} = \overline{\left(I^{(c)}\right)^{\ell k}} = \overline{J^{\ell k}}$. This, together with Theorem \ref{thm.c} again, implies that
	$$I^{(c\ell k)} = \left(\overline{J^\ell}\right)^k = \left[\overline{(I^{(c)})^{\ell}}\right]^k = \left(I^{(c \ell)}\right)^k \ \text{ for all } \ k \in \NN.$$
	It then follows from the definition of $\svd(I)$ that $\svd(I) \le c\ell$, and the (2) is proved.
\end{proof}

The following example illustrates that the bounds in Theorem \ref{thm.svd} are sharp.

\begin{example}
	\label{ex.svdsharp}
	Let $I = (x,y)^2 \cap (y,z)^3 \cap (z,x)^4$ be the monomial ideal as in Example \ref{ex.cFromSP1}. As computed in Example \ref{ex.cFromSP1}, we have $c = 2$. All the vertices of $\SP(I)$ were also given in Example \ref{ex.cFromSP1}. Thus, we get $\ell_s(I) = 2$. Hence, Theorem \ref{thm.svd} forces $\svd(I) = 2$ and we have the equality $c = \svd(I) = (\ell_s(I)-1)c.$
\end{example}

Clearly, monomial ideals that have standard graded symbolic Rees algebras satisfy $c=\svd(I)=1$.



In \cite[Question 3.10]{GS2021}, it is asked if for all monomial ideals $I$, $\svd(I)$ is at most the least common multiple of the algebra generating degrees of the symbolic Rees algebra $\R_s(I)$. We provide a weaker statement for squarefree monomial ideals.

\begin{theorem}
	\label{thm.boundC}
	Let $I$ be a monomial ideal of linear-power type and let $c$ be obtained as in Theorem \ref{thm.c}. Then, $c$ divides the least common multiple of the algebra generating degrees of the symbolic Rees algebra $\R_s(I)$.
\end{theorem}

\begin{proof} Let $\{v_1, \dots, v_r\}$ be the vertices of $\SP(I)$. For $i = 1, \dots, r,$ let $d_i$ be the least common multiple of the denominators of the coordinates of $v_i$. By definition, $c = \lcm (d_1, \dots, d_r)$. It suffices to show that $d_i$ is an algebra generating degree of the symbolic Rees algebra $\R_s(I)$, for all $i = 1, \dots, r$.
	
	Let $\text{SC}(I)$ be the \emph{Simis cone} of $I$ as constructed in \cite[(4)]{EVY2016}. It can be seen from the definition that $\text{SC}(I)$ is the cone over $\SP(I)$ in $\RR^{n+1}$, where $\SP(I)$ is viewed as lying in the hyperplane $x_{n+1} = 1$ (that is isomorphic to $\RR^n$). Thus, $\{(v_1,1), \dots, (v_r,1)\}$ are the extremal rays of $\text{SC}(I)$. This particularly implies that $(d_iv_i, d_i)$ cannot be written as the sum of two integral points in $\text{SC}(I)$, for all $i = 1, \dots, r$. In other words, $(d_iv_i, d_i)$ is in the integral Hilbert basis of $\text{SC}(I)$, for all $i = 1, \dots, r$. It now follows from \cite[Corollary 3.2]{MRV2011} that $d_i$ is an algebra generating degree of $\R_s(I)$ for all $i = 1, \dots, r$. The theorem is proved.
\end{proof}

\begin{example}
	\label{ex.oddcycle}
	Let $I = I(C_{2n+1}) \subseteq \kk[x_1, \dots, x_{2n+1}]$ be the edge ideal of an odd cycle of length $2n+1$. It follows from \cite[Theorem 3.4]{GHOS2020} that $\R_s(I) = R[It, (x_1 \dots x_{2n+1}) t^{n+1}]$. Thus, $\svd(I) = n+1$. Observe further that the vertices of $\SP(I)$ are integral vertices (of $\NP(I)$) together with $(\frac{1}{n+1}, \dots, \frac{1}{n+1})$. Therefore, $c = n+1$. In this example, we have $\svd(I) = c$ is equal to the least common multiple of the minimal generating degrees of the symbolic Rees algebra of $I$.
\end{example}

We shall give another example illustrating that the bound for $c$ in Theorem \ref{thm.boundC} is sharp. To do so, we will need a lemma.

\begin{lemma}
	\label{lem.uniqueSol}
	Let $b < n$ be positive integers. Consider the following system of linear equations: $x_{i_1} + \dots + x_{i_b} = 1$, for all choices of $1 \le i_1 < \dots < i_b \le n$. In $\RR^n_{\ge 0}$, this system has a unique solution given by $x_1 = \dots = x_n = \dfrac{1}{b}.$
\end{lemma}

\begin{proof} It is obvious that $x_1 = \dots = x_n = \frac{1}{b}$ is a solution to the given system. We shall prove that it is the unique solution in $\RR^n_{\ge 0}$. Consider an arbitrary solution $(t_1, \dots, t_n) \in \RR^n_{\ge 0}$ to the given system. Without loss of generality, we may assume that $t_1, \dots, t_s < \frac{1}{b}$, while $t_{s+1}, \dots, t_n \ge \frac{1}{b}$, for some $s < n$. If $s = 0$ then the assertion is proved. Suppose that $s \ge 1$.
	
	It is easy to see that $s \le b-1$; otherwise, $t_1 + \dots + t_b < \frac{1}{b}\cdot b = 1$ is a contradiction. Thus, $t_{b+1} \ge \frac{1}{b} > t_1$. Now, we have
	$$1 = t_1 + \dots + t_b < t_2 + \dots + t_{b+1} = 1,$$
	a contradiction. The lemma is proved.
\end{proof}

\begin{example}
	\label{ex.starConfiguration}
	Let $R = \kk[x_1, \dots, x_n]$ and let $h < n$. Let $I$ be the following codimension $h$ star configuration
	$$I = \bigcap_{1 \le i_1 < \dots < i_h \le n} (x_{i_1}, \dots, x_{i_h}).$$
	Let $c$ be as in Theorem \ref{thm.c}; that is, $c$ is the least integer such that, for all $k \in \NN$,
	$$I^{(ck)} = \overline{\left(I^{(c)}\right)^k}.$$
	We claim that $c = \lcm(1,2, \dots, h).$ Note that when $h=n$, $I^k=\overline{I^k}=I^{(k)}$ for all $k$, so $c=1$.
	
	From the description of $\SP(I)$ in (\ref{eq.SP}), we know that the defining half-spaces of $\SP(I)$ are $x_i \ge 0$, for $i = 1, \dots, n$, and $x_{i_1} + \dots + x_{i_h} \ge 1$, for all $1 \le i_1 < \dots < i_h \le n$. Let $S$ be the set of collections $\{i_1, \dots, i_h\}$ with $1 \le i_1 < \dots < i_h \le n$, and set $[n] = \{1, \dots, n\}$. Consider a vertex $v = (a_1, \dots, a_n)$ of $\SP(I)$. Then, $v$ is determined by the unique solution to the following system of inequality, for some subsets $A \subseteq [n]$ and $B \subseteq S$,
	\begin{align}
		\left\{ \begin{array}{lll} x_i = 0 & \text{for} & i \in A \\
		x_i \ge 0 & \text{for} & i \not\in A \\
		\sum_{i \in T} x_i = 1 & \text{for} & T \in B \\
		\sum_{i \in T} x_i \ge 1 & \text{for} & T \in S \setminus B.\end{array}\right. \label{eq.systemIE}
	\end{align}

We will show that for any $T \in B$, $T$ must contain $A$. Indeed, suppose, by contradiction, that there exist $T \in B$ and $i \in A$ such that $x_i \not\in T$. Without loss of generality, assume that $T = \{x_2, \dots, x_{h+1}\}$ and $x_1 \in A \setminus T$.	Then, $a_1 = 0$ and $\sum_{i=2}^{h+1} a_i = 1$. Furthermore, by looking at the other in equality of the system, we also have
$$a_1 + \sum_{\substack{2 \le i \le h+1 \\ i \not= j}} a_i \ge 1 \text{ for } j = 2, \dots, h+1.$$
This implies that $\sum_{i=2}^{h+1} a_i \ge \frac{h}{h-1} > 1$, a contradiction.

Let $a = |A|$. It follows from what we just showed that $(b_1, \dots, b_n)$, with $b_i = 0$ if $i \in A$ and $b_i = \frac{1}{h-a}$ if $i \not\in A$, is a solution to the system (\ref{eq.systemIE}) and, thus, it is the unique solution giving $v$. Particularly, this implies the least common multiple of the denominators of the coordinates in $v$ is $h-a$ for some $a$.

Conversely, for any $0 \le a \le h-1$, consider
$$u = \Big(\underbrace{0, \dots, 0}_{a \text{ times}}, \underbrace{\frac{1}{h-a}, \dots, \frac{1}{h-a}}_{n-a \text{ times}}\Big).$$
Let $A = \{1, \dots, a\}$ and let $B = \{T \in S ~\big|~ A \subseteq T\}$. Consider the system of inequality (\ref{eq.systemIE}) with these choices for $A$ and $B$. Let $v = (a_1, \dots, a_n)$ any solution in $\RR^n_{\ge 0}$ to this system of inequality. Obviously, we have $a_i = 0$ for all $i = 1, \dots, a$. On the other hand, by considering any subset $A'$ in $[n]$ of size $h-a$ that is disjoint from $A$ then since $\sum_{i \in A \cup A'} a_i = 1$, we have that $\sum_{i \in A'} a_i = 1$. Thus, by Lemma \ref{lem.uniqueSol}, we must have $a_{a+1} = \dots = a_n = \frac{1}{h-a}$. We have just shown that $u$ is the unique solution to (\ref{eq.systemIE}) with the chosen $A$ and $B$. Particularly, $u$ is a vertex of $\SP(I)$. Hence, for any $0 \le a \le h-1$, there is a vertex $u$ of $\SP(I)$ for which the least common multiple of the denominators of its coordinate is precisely $h-a$.

We have established the claim that $c = \lcm(1, \dots, h)$.
On the other hand, by \cite[Example 3.9]{GS2021}, it is known that the symbolic Rees algebra $\R_s(I)$ is generated in degrees $1, \dots, h$. This example, therefore, exhibits that the bound for $c$ in Theorem \ref{thm.boundC} is sharp.
\end{example}

\begin{remark} Theorem \ref{thm.boundC} may suggest that the vertices of $\SP(I)$ give the Hilbert basis for the Simis cone $\text{SC}(I)$, and determine the generating degrees of the symbolic Rees algebra $\R_s(I)$ of $I$. This is however not the case. Consider the ideal $I$ as in Example \ref{ex.cFromSP2}. As shown, $\SP(I)$ has 17 vertices. On the other hand, it can be checked that $abcdef \in I^{(4)}$ gives a minimal generator $g = abcdef t^4$ for the symbolic Rees algebra $\R_s(I)$. In fact, the integral Hilbert basis of $\text{SC}(I)$ has precisely 18 elements, 17 of which come from the vertices of $\SP(I)$ and the last one comes from $g$.
\end{remark}

We continue to present our next result that gives a bound for the symbolic generation type of a monomial ideal. We shall need a lemma whose proof is in the same spirit as that given in \cite[Theorem 5.1]{Sin2007}.

\begin{lemma}
	\label{lem.Singla}
	Let $I$ be a monomial ideal and let $\SP(I)$ be its symbolic polyhedron. For any point $v \in \SP(I)$, we can write $v = u+w$, where $u$ is a point on a compact face of $\SP(I)$ and $w \in \RR^n_{\ge 0}$.
\end{lemma}

\begin{proof} Suppose that $v = (a_1, \dots, a_n) \in \SP(I)$. Since $\SP(I) \subseteq \RR^n_{\ge 0}$, there exists at least one coordinate $a_j > 0$ of $v$ such that for $\lambda \gg 0$, $v' = (a_1, \dots, a_{j-1}, a_j - \lambda, a_{j+1}, \dots, a_n) \not\in \SP(I)$. If $v$ is not on any face of $\SP(I)$ then let $v''$ be the intersection of the boundary of $\SP(I)$ and the line segment connecting $v$ and $v'$. We have
	\begin{enumerate}
		\item[(1)] $v''$ is on a face $F$ of $\SP(I)$ (not necessarily compact); and
		\item[(2)] $v = v'' + w$ for some $w \in \RR^n_{\ge 0}$.
	\end{enumerate}
Thus, by replacing $v$ with $v''$, it suffices to prove the assertion when $v$ is on a face $F$ of $\SP(I)$ and $F$ is not compact.

Let $H_{\p,k}$ be the supporting hyperplane of $F$. That is, $F = \SP(I) \cap H_{\p,k}$ and $H_{\p,k} = \{\q \in \RR^n ~\big|~ \langle \p ,\q\rangle = k\}$ for some fixed vector $\p \in \RR^n$ and $k \in \ZZ$. We shall use induction on $\dim F$. If $\dim F = 1$ then $v$ can be written as $v = v_j + \b$ for some vertex $v_j$ of $\SP(I)$ and $\b \in \RR^n_{\ge 0}$. The assertion is straightforward in this case. Assume that $\dim F \ge 2$.

Since $F$ is not compact, there exists a coordinate $p_j$ of $\p$ that is 0. Again, for $\lambda \gg 0$, $v' = (a_1, \dots, a_{j-1}, a_j - \lambda, a_{j+1}, \dots, a_n) \not\in \SP(I)$. Observe that, since $p_j = 0$, the line segment connecting $v$ and $v'$ lies in $H_{\p,k}$. Let $v''$ be the intersection of $F$ and the line segment connecting $v$ and $v'$. Then, as before, we have
\begin{enumerate}
	\item[(3)] $v''$ lies in a proper face $F'$ of $F$;
	\item[(4)] $v = v'' + w$ for some $w \in \RR^n_{\ge 0}$.
\end{enumerate}

If $F'$ is compact then the assertion follows. If $F'$ is not compact then by the induction hypothesis, we can write $v'' = u+w'$, where $u$ belong to a compact face of $\SP(I)$ and $w' \in \RR^n_{\ge 0}$. Now, $v = u + (w + w'')$ and the proof completes.
\end{proof}

\begin{theorem}
	\label{thm.GT}
	Let $I$ be a monomial ideal of linear-power type. Suppose that $\{v_1, \dots, v_r\}$ are the vertices of $\SP(I)$. Let $d_i$ be the least common multiple of the denominator of coordinates of $v_i$, for $i = 1, \dots, r$. Set $D = \max\{d_1, \dots, d_r\}$. Then,
	$$\sgt(I) \le \max\{\ell_s(I)D-1, D\}).$$
\end{theorem}

\begin{proof} Consider any minimal algebra generator $x^\a t^k$ of the symbolic Rees algebra $\R_s(I)$. Particularly, this implies that $x^\a \in I^{(k)}$. By \cite[Corollary 3.2]{MRV2011}, $(\a,k) \in \RR^{n+1}$ is a element of the integral Hilbert basis for the Simis cone $\text{SC}(I)$ of $I$. By Lemma \ref{lem.membership}, we also have $\frac{\a}{k} \in \SP(I)$.
	
	It follows from Lemma \ref{lem.Singla} that $\frac{\a}{k}$ can be written as $\frac{\a}{k} = u+w$, where $u$ belongs to a compact face $F$ of $\SP(I)$ and $w \in \RR^n_{\ge 0}$. Corollary \ref{cor.ellsG} implies that $\dim F \le \ell_s(I)-1$. Thus, by the Caratheodory theorem, there exist at most $\ell_s(I)$ vertices $v_{i_1}, \dots, u_{i_\ell}$, for some $\ell \le \ell_s(I)$, and constants $\lambda_1, \dots, \lambda_\ell \in [0,1]$ such that $\sum_{j=1}^\ell \lambda_j = 1$ and
	$u = \sum_{j=1}^\ell \lambda_jv_{i_j}.$
	Set $m_j = k\lambda_j$, for $j = 1, \dots, \ell$. We get $\sum_{j=1}^\ell m_j = k$ and
	$$\a = ku+kw = \sum_{j=1}^\ell m_jv_{i_j} + kw.$$
	
	Suppose that $k > \max\{\ell_s(I)D-1, D\}$. Then, $\sum_{j=1}^\ell m_j = k > \ell_s(I)D.$ Therefore, there exists $j$ such that $m_j \ge D \ge d_{i_j}$. This implies that
	$$\a - d_{i_j}v_{i_j} = \sum_{s \not= j} m_sv_{i_s} + (m_{i_j}-d_{i_j})v_{i_j} + kw.$$
	It follows, since $\sum_{s\not= j}m_s + (m_{i_j}-d_{i_j}) = k-d_{i_j}$, that
	$$\dfrac{\a-d_{i_j}v_{i_j} - kw}{k-d_{i_j}} \in \SP(I).$$
	Particularly, we have $(\a-d_{i_j}v_{i_j}-kw, k-d_{i_j}) \in \text{SC}(I)$, and so $(\a-d_{i_j}v_{i_j},k-d_{i_j}) \in \text{SC}(I)$. Now we can write
	$$(\a,k) = (\a-d_{i_j}v_{i_j},k-d_{i_j}) + (d_{i_j}v_{i_j},d_{i_j}),$$
	which is a contradiction to the fact that $(\a,k)$ is in the integral Hilbert basis of $\text{SC}(I)$ or, equivalently, $x^\a t^k$ is a minimal generator of $\R_s(I)$. The theorem is proved.
\end{proof}

\begin{remark}
	As observed in Example \ref{ex.starConfiguration}, the vertices of $\SP(I)$ give rise to elements in the integral Hilbert basis for the Simis cone $\text{SC}(I)$. Thus, it follows that $\sgt(I) \ge D$. If $\ell_s(I) \ge 2$ then $\ell_s(I)D-1 \ge D$, so the bound of Theorem \ref{thm.GT} in this case simply reads
	$$\sgt(I) \le \ell_s(I)D-1.$$
	The constant $D$ was included in the right hand side to take care of the case when $\ell_s(I) = 1$.
\end{remark}



Theorem \ref{thm.GT} gives several interesting corollaries. The first part of the following corollary slightly improves \cite[Theorem 5.6]{HHT2007} (since $\ell_s(I) \le n$ and $\ell_s(I) \le n-1=\dim(S)-1$ if $\mm \not \in \Ass(I)$).

\begin{corollary} \label{cor.boundSGT}
	Let $I$ be as in Theorem \ref{thm.GT}. Then
	\begin{enumerate}
		\item $\sgt(I) \le \max\left\{\ell_s(I)\dfrac{(n+1)^{(n+1)/2}}{2^n}-1, \dfrac{(n+1)^{(n+1)/2}}{2^n}\right\}$; and
		\item $\sgt(I) \le \max\{\ell_s(I)c-1, c\}.$
	\end{enumerate}
\end{corollary}

\begin{proof} Observe that each vertex $v_j$ of $\SP(I)$ is given by the unique solution of a system of the form (\ref{eq.systemIE}). Particularly, the denominator of any coordinate of $v_j$ and, hence, $d_i$ divides the determinant of an $n \times n$ matrix with entries being 0 and 1. Such a determinant has a classical bound (see, for instance \cite{FS1965}) given by $\dfrac{(n+1)^{(n+1)/2}}{2^n}.$ This proves (1). Part (2) is straightforward since $c \ge D$.
\end{proof}

With essentially the same proof as that of Theorem \ref{thm.GT}, we obtain the following result which improves \cite[Corollary 3.11]{EVY2016} and implies \cite[Corollary 5.3]{Sin2007} (see also \cite{RRV2003}) for linear-power typed ideals.

\begin{corollary}
	\label{cor.generatingdegree}
	Let $I$ be a monomial ideal of linear-power type. Then, $\overline{\R(I)}$ is minimally generated in degrees at most $\max\{\ell(I)-1, 1\}$. In particular, if $I^k = \overline{I^k}$ for $1 \le k \le \ell(I)-1$ then $I^k = \overline{I^k}$ for all $k \in \NN$.
\end{corollary}

\begin{proof} If $\ell(I) = 1$ then $\mdc(\NP(I)) = 0$, and so $\NP(I)$ has only one vertex. This implies that $I$ is principal. Particularly, we have $I^k = \overline{I^k}$ for all $k \in \NN$. The assertion in this case is obvious. Suppose that $\ell(I) \ge 2$. The proof proceeds in the same line of arguments as that of Theorem \ref{thm.GT}.
	
	Consider any minimal generator $x^\a t^k$ of $\overline{\R(I)} = R \oplus \overline{I}t \oplus \overline{I^2}t^2 \oplus \dots$. We have $x^\a \in \overline{I^k}$, and so it follows from Lemma \ref{lem.membership} that $\frac{\a}{k} \in \NP(I)$. By the same proof as that of Lemma \ref{lem.Singla}, replacing $\SP(I)$ with $\NP(I)$, we can write $\frac{\a}{k} = u + w$, where $u$ belongs to a compact face $G$ of $\NP(I)$ and $w \in \RR^n_{\ge 0}$.
	
	Observe that, by \cite[Theorem 2.3]{BA2003}, we have $\dim G \le \ell(I) - 1$. Thus, again by Caratheodory theorem, there exist as most $\ell(I)$ vertices $v_{i_1}, \dots, v_{i_\ell}$ of $G$, for some $\ell \le \ell(I)$, and $\lambda_1, \dots, \lambda_\ell \in [0,1]$ such that $\sum_{j =1}^\ell \lambda_j = 1$ and
	$u = \sum_{j=1}^\ell \lambda_j v_{i_j}.$ Set $m_j = k\lambda_j$, for $j = 1, \dots, \ell$. We get $\sum_{j=1}^\ell m_j = k$ and
	$$\a = ku+kw = \sum_{j=1}^\ell m_jv_{i_j} + kw.$$
	
	Suppose that $k \ge \ell(I)$. Then, there exists an index $j$ such that $m_j \ge 1$. Now, similarly to the computation in Theorem \ref{thm.GT}, we get $\frac{\a-v_{i_j}}{k-1} \in \NP(I)$. That is, $x^{\a-v_{i_j}} \in \overline{I^{k-1}}$, and so $x^\a t^k = x^{\a-v_{i_j}} t^{k-1} \cdot x^{v_{i_j}} t$ is not a minimal generator of $\overline{\R(I)}$, a contradiction. Hence, $k \le \ell(I)-1$ and the theorem is proved.
\end{proof}

As a consequence of Corollary \ref{cor.generatingdegree}, if $I$ is a monomial ideal of linear-power type such that $\SP(I) = \NP(I)$, then we have
$$\sgt(I) \le \max\{\ell_s(I)-1,1\} = \max \{\ell(I)-1,1\}.$$
This bound can be slightly improved for squarefree monomial ideals, as shown in the following corollary.

\begin{corollary}
	\label{cor.boundGT2}
	Let $I$ be a squarefree monomial ideal such that $\NP(I) = \SP(I)$. Then,
	$$\sgt(I) \le \max\{\ell_s(I)-2, 1\} = \max\{\ell(I)-2,1\}.$$
\end{corollary}

\begin{proof} If $\ell(I) = 1$ or $\ell(I) = 2$ then the assertion follows directly from Corollary \ref{cor.generatingdegree}. Suppose that $\ell(I) \ge 3$.
It suffices to show that $\overline{\R(I)}$ is minimally generated in degrees at most $\ell(I)-2$.
	
Consider a minimal generator $x^\a t^k$ of $\overline{\R(I)}$. As before, we have $\frac{\a}{k} \in \NP(I)$ that can be written as $\frac{\a}{k} = u+w$, where $u$ belongs to a compact face $G$ of $\NP(I)$ and $w \in \RR^n_{\ge 0}$. Furthermore, there exist $\ell \le \ell(I)$ vertices $v_{i_1}, \dots, v_{i_\ell}$ of $G$ and $m_j \in [0,k]$ such that $\sum_{j=1}^\ell m_j =k$ and
	$$\a = \sum_{j=1}^\ell m_jv_{i_j} + kw.$$

If there exists $m_j \ge 1$ then, similar to the arguments in Corollary \ref{cor.generatingdegree}, we derive at a contradiction to the minimality of $x^\a t^k$. Thus, $0 \le m_j < 1$ for all $j = 1, \dots, \ell$.
By induction on $n$ and localization at the variables corresponding to zero entries in $\a$, we may further assume that all coordinates of $\a$ are nonzero.

Let $e_1, \dots, e_n$ represent the unit vectors in $\RR^n$ and, for $\pp = (x_{i_1}, \dots, x_{i_s}) \in \Ass(I)$, set $v_\pp = e_{i_1} + \dots + e_{i_s}$. It follows from Lemma \ref{lem.membership} and (\ref{eq.SP}) that $\langle \a, v_\pp\rangle \ge k$ for all $\pp \in \Ass(I)$.

If $\langle \a , v_\pp \rangle \ge k+1$ for all $\pp \in \Ass(I)$ then $\langle \a-e_1, v_\pp\rangle \ge k$ for all $\pp \in \Ass(I)$, and so $x^{\a-e_1} t^k \in \SP(I) = \NP(I)$, which is a contradiction to the minimality of $x^\a t^k$. Therefore, there exists a nonempty subset $A$ of $\Ass(I)$ such that
$$\langle \a, v_\pp\rangle = k \text{ if and only if } \pp \in A.$$
Let $H_\pp$ be the hyperplane defined by $\langle \a, v_\pp\rangle = 0$, for $\pp \in \Ass(I)$, and set $H_A = \bigcap_{\pp \in A} H_\pp$.

If there exists a unit vector $e_j$ lying in $H_A$ then it can be seen that $\langle \a-e_j, v_\pp\rangle \ge k$ for all $\pp \in \Ass(I)$. Thus, $x^{\a-e_j}t^k \in \SP(I) = \NP(I)$, which is again a contradiction to the minimality of $x^\a t^k$. We may assume that \emph{none} of the unit vectors $e_1, \dots, e_n$ is on $H_A$.

Observe moreover that, since $v_{i_j}$ is a vertex of $\SP(I)$, we have $\langle v_{i_j}, v_\pp\rangle \ge 1$ for all $\pp \in \Ass(I)$. It follows that
	$$\langle \a, v_\pp\rangle  = \sum_{j=1}^\ell m_j \langle v_{i_j}, v_\pp\rangle + k\langle w, v_\pp\rangle
	 \ge \sum_{j=1}^\ell m_j + k\langle w, v_\pp\rangle = k + k\langle w,v_\pp\rangle.$$
This implies that, for all $\pp \in A$, $\langle w,v_\pp\rangle = 0$. Particularly, if the coefficient of a unit vector $e_j$ in $w$ is nonzero then $\langle e_j, v_\pp \rangle = 0$ for all $\pp \in A$ or, equivalently, $e_j \in H_A$, which is a contradiction to our assumption. We conclude that $w = 0$.

Since $I$ is a squarefree monomial ideal and $I$ is not principal (otherwise, $\ell(I) = 1$), $v_{i_1}$ must contain a coordinate being 0. Without loss of generality, assume that the first coordinate of $v_{i_1}$ is 0, and among $v_{i_1}, \dots, v_{i_\ell}$ there are exactly $s$ vertices with nonzero first coordinates, for some $s < \ell$. Let $\a = (a_1, \dots, a_n)$. Then, ${\displaystyle a_1 = \sum_{v_{i_j} \text{ has nonzero first coordinate}} m_j < s}$. Since $a_1 \in \ZZ$, we have $a_1 \le s-1$. It follows that
$$k = \sum_{j=1}^\ell m_j = \sum_{v_{i_j} \text{ has nonzero first coordinate}} m_j + \sum_{v_{i_j} \text{ has 0 first coordinate}} m_j < s-1 + \ell-s = \ell-1.$$
Hence, we conclude that $k \le \ell-2$.
\end{proof}

The following example shows that the bound in Corollary \ref{cor.boundGT2} is sharp.

\begin{example}
	\label{ex.GT2sharp}
	Let $I = (abc, aef, bdf, cde) = (a,d) \cap (b,e) \cap (c,f) \cap (a,b,c) \cap (a,e,f) \cap (b,d,f) \cap (c,d,e) \subseteq \kk[a, \dots, f]$.
	A direct computation shows that $\SP(I)$ has vertices
	$$\{(1,1,1,0,0,0), (1,0,0,0,1,1), (0,1,0,1,0,1), (0,0,1,1,1,0)\} \subseteq \RR^6.$$
	Particularly, this implies that $c = 1$, and so by Theorem \ref{thm.c}, we have $\SP(I) = \NP(I)$.
	
	On the other hand, the integral Hilbert basis of the Simis cone $\text{SC}(I)$ contains an additional vertex $(1,1,1,1,1,1,2) \in \RR^7$. Thus, $\sgt(I) = 2$. Furthermore, the 4 vertices of $\SP(I)$ form a compact face of dimension 3. Therefore, by Corollary \ref{cor.ellsG}, we get $\ell_s(I)= \ell(I) = 4$. Hence, in this example, $\sgt(I) = \ell_s(I)-2$, exhibiting that the bound in Corollary \ref{cor.boundGT2} is sharp.
\end{example}

\bibliographystyle{alpha}
\bibliography{Reference}

\end{document}